\documentclass[aap,preprint]{imsart}
\usepackage[letterpaper]{geometry}
\usepackage{amsthm,amsmath,amsfonts}
\usepackage{imsart}
\usepackage{color,graphicx}
\usepackage[colorlinks,linkcolor=blue,citecolor=cyan,anchorcolor=blue]{hyperref}
\usepackage[numbers,sort]{natbib}
\usepackage{enumerate}
\usepackage{wrapfig}
\newcommand{\R}{\mathbb{R}}
\newcommand{\D}{\mathbb{D}}

\newcommand{\Z}{\mathbb{Z}}
\newcommand{\E}{\mathbb{E}}
\newcommand{\Prob}{\mathbb{P}}

\providecommand{\abs}[1]{\left\lvert#1\right\rvert}

\newtheorem{lemma}{Lemma}
\newtheorem{theorem}{Theorem}
\newtheorem{corollary}{Corollary}
\theoremstyle{definition}
\numberwithin{equation}{section}

\newtheorem{remark}{Remark}
\arxiv{arXiv:1503.00774}

\begin{document}
\begin{frontmatter}

\title{Stein's method for steady-state diffusion approximations
  of $M/Ph/n+M$ systems\protect\thanksref{T1}}
\runtitle{Stein's method for diffusion approximations}
\thankstext{T1}{This research is supported in part by NSF Grants CMMI-1030589, CNS-1248117 and CMMI-1335724.}

\begin{aug}
  \author{\fnms{Anton}  \snm{Braverman}\corref{}\ead[label=e1]{ab2329@cornell.edu}}
  \and
  \author{\fnms{J. G.} \snm{Dai}\ead[label=e2]{jd694@cornell.edu}}

  \runauthor{A. Braverman and J. G. Dai}

  \affiliation{Cornell University}

  \address{School of Operations Research\\and Information Engineering\\Cornell University\\Ithaca, NY 14850\\USA\\ 
          \printead{e1,e2}}

\end{aug}

\begin{abstract}
  We consider $M/Ph/n+M$ queueing systems in steady state. We
  prove that the Wasserstein distance between the stationary
  distribution of the normalized system size process and that of a
  piecewise Ornstein-Uhlenbeck (OU) process is bounded by
  $C/\sqrt{\lambda}$, where the constant $C$ is independent of the
  arrival rate $\lambda$ and the number of servers $n$ as long as they
  are in the Halfin-Whitt parameter regime. For each integer $m>0$, we
  also establish a similar bound for the difference of the $m$th
  steady-state moments. For the proofs, we develop
  a modular framework that is based on Stein's method. The framework
  has three components: Poisson equation, generator coupling, and
  state space collapse. The framework, with further refinement, is
  likely applicable to steady-state diffusion approximations for other
  stochastic systems.
\end{abstract}



\end{frontmatter}

\section{Introduction}

This paper focuses on $M/Ph/n+M$ systems, which serve as building
blocks to model large-scale service systems such as customer contact
centers \cite{GansKoolMand2003,AksiArmoMehr2007} and hospital
operations \cite{Armoetal2011,Shietal2014}.  In such a system, there
are $n$ identical servers, the arrival process is Poisson (the symbol
$M$) with rate $\lambda$, the service times are i.i.d.\ having a
phase-type distribution (the symbol $Ph$) with mean $1/\mu$, the
patience times of customers are i.i.d.\ having an exponential
distribution (the symbol $+M$) with mean $1/\alpha<\infty$. When the
waiting time of a customer in queue exceeds her patience time, the
customer abandons the system without service; once the service of a
customer is started, the customer does not abandon.

Let $X_i(t)$ be the number of customers in phase $i$ at time $t$ for
$i=1, \ldots, d$, where $d$ is the number of phases in the service
time distribution. Let $X(t)$ be the corresponding vector. Then the
system size process $X=\{X(t), t\ge 0\}$ has a
unique stationary distribution for any arrival rate $\lambda$ and any
server number $n$ due to customer abandonment; although $X$ is not a Markov chain, it is a function of a Markov chain with a unique stationary distribution,  see Section~\ref{sec:CTMCrep} for details.  In this paper, we
prove, in Theorem \ref{thm:main}, that
\begin{equation} \label{eq:introresult}
\sup \limits_{h \in \mathcal{H}}  \abs{\E \big[h(\tilde X^{(\lambda)}(\infty))\big] - \E \big[h(Y(\infty))\big]} \leq \frac{C}{\sqrt{\lambda}} \quad \text{for any } \lambda > 0 \text{ and } n \ge 1 
\end{equation}
satisfying
\begin{equation}
  \label{eq:square-root}
  n \mu = \lambda + \beta \sqrt{\lambda},
\end{equation}
where $\beta\in \R$ is some constant and $\mathcal{H}$ is some class
of functions $h:\R^d\to \R$.  In (\ref{eq:introresult}), $\tilde
X^{(\lambda)}(\infty)$ is a random vector having the stationary distribution of
a properly scaled version of $X = X^{(\lambda)}$ that depends on the arrival rate $\lambda$, number of servers $n$, the service time distribution, and the abandonment rate $\alpha$, and $Y(\infty)$ is a random vector
having the stationary distribution of a piecewise Ornstein-Uhlenbeck
(OU) process $Y=\{Y(t), t\ge 0\}$. The stationary distribution of $X^{(\lambda)}$ exists even when $\beta$ is negative because $\alpha$ is assumed to be positive. The constant $C$ depends on the
service time distribution, abandonment rate $\alpha$, the  constant $\beta$ in (\ref{eq:square-root}), and the choice of ${\cal
  H}$, but $C$ is independent of the arrival rate $\lambda$ and the number of servers $n$. Two different classes $\mathcal{H}$ will used in our Theorem
\ref{thm:main}. First, we take $\cal{H}$ to be the class of polynomials up to a certain
order. In this case, (\ref{eq:introresult}) provides rates of
convergence for steady-state moments. Second, ${\cal H}$ is taken to be ${\cal W}^{(d)}$, the class of all $1$-Lipschitz functions
\begin{equation}
  \label{eq:Wd}
     {\cal W}^{(d)} = \{h: \R^d\to\R: \abs{h(x)-h(y)}\le \abs{x-y}\}.
\end{equation}
In this case, (\ref{eq:introresult}) provides rates of convergence for
stationary distributions under the Wasserstein metric \cite{Ross2011};
convergence under Wasserstein metric implies the convergence in
distribution \cite{GibbSu2002}.

In \cite{DaiHe2013}, an algorithm was developed to compute the
stationary distribution of the diffusion process $Y$.  The
distribution of $Y(\infty)$ is then used to approximate the stationary
distribution of $X^{(\lambda)}$. The approximation is remarkably accurate; see,
for example, Figure 1 there. It was demonstrated that computational efficiency, in terms of both time and memory, can be
achieved by diffusion approximations. For example, in an $M/H_2/500+M$
system studied in \cite{DaiHe2013}, where the system has $500$ servers
and a hyper-exponential service time distribution, it took around 1
hour and peak memory usage of 5 GB to compute the stationary distribution of
$X^{(\lambda)}$ using an algorithm that fully explores the special structure of a
three-dimensional Markov chain.  On the same computer, to compute the
stationary distribution of the corresponding two-dimensional
diffusion process it took less
than 1 minute and peak memory usage was less than 200 MB.  The computational saving by the diffusion model is
achieved partly through \emph{state space collapse} (SSC), a
phenomenon that causes dimension reduction in state space.
Theorem \ref{thm:main} quantifies the steady-state diffusion approximations developed in \cite{DaiHe2013}.

In  \cite{GurvHuanMand2014}, the authors prove a version of (\ref{eq:introresult}) for the $M/M/n+M$ system,
a special case of the $M/Ph/n+M$ system where the service time
distribution is exponential. They do not impose assumption
\eqref{eq:square-root} on the relationship between the arrival rate
$\lambda$ and number of servers $n$, resulting in a \textit{universal}
approximation that is accurate in any parameter regime, from
underloaded, to critically loaded, and to overloaded.  
To our knowledge, this is the first paper to study
convergence rates of steady state diffusion approximations.
Their method
relies on analyzing excursions of a one-dimensional Markov chain and 
the
corresponding diffusion process. It is unclear how to generalize 
their method to the multi-dimensional setting.


To prove Theorem \ref{thm:main}, we develop a framework that is based
on Stein's method \cite{Stei1972,Stei1986}.  The framework is modular
and relies on three components: a Poisson equation, generator
coupling, and SSC. The framework itself is an important part of our
contribution, in addition to Theorem \ref{thm:main}. We expect the
framework will be refined and used to prove rates of convergence of
steady-state diffusion approximations for many other stochastic
systems. This framework is closely related to a recent paper
\cite{Gurv2014} by Gurvich. We will discuss his work after giving an
overview of the framework.


We consider two sequences of
stochastic processes $\{X^{(\ell)}\}_{\ell =1}^{\infty}$ and
$\{Y^{(\ell)}\}_{\ell=1}^{\infty}$ indexed by $\ell$, where $X^{(\ell)} = \{X^{(\ell)}(t)\in
\R^d, t \geq 0\}$ is a continuous-time Markov chain (CTMC) and $Y^{(\ell)} = \{Y^{(\ell)}(t) \in \R^d, t \geq 0\}$
is a diffusion process.
Suppose $X^{(\ell)}(\infty)$ and
$Y^{(\ell)}(\infty)$ are two random vectors having the stationary
distributions of $X^{(\ell)}$ and $Y^{(\ell)}$, respectively. 
Let $G_{X^{(\ell)}}$ and $G_{Y^{(\ell)}}$ be the generators of $X^{(\ell)}$ and $Y^{(\ell)}$,
respectively;  for a diffusion process, $G_{Y^{(\ell)}}$ is the second order
elliptic operator as in \eqref{eq:DMgen}.
For a function $h: \R^d \to \R$ in a "nice" (but
large enough) class, we wish to bound
\begin{displaymath}
\abs{\E h(X^{(\ell)}(\infty)) - \E h(Y^{(\ell)}(\infty))}.
\end{displaymath}

\textbf{Component 1.} The first step is to set up the Poisson equation 
\begin{equation} \label{eq:intropoisson}
G_{Y^{(\ell)}} f_h(x) = h(x) - \E h(Y^{(\ell)}(\infty))
\end{equation}
and obtain various estimates of a solution $f_h$ to the Poisson
equation. Once we have $f_h$, one can take the expectation of both sides above to see that
\begin{equation}
  \label{eq:PoissonRep}
\E h(X^{(\ell)}(\infty)) - \E h(Y^{(\ell)}(\infty)) = \E G_{Y^{(\ell)}} f_h(X^{(\ell)}(\infty)).  
\end{equation}
 The Poisson equation \eqref{eq:intropoisson} is a partial
differential equation (PDE).  Even when $Y^{(\ell)}(\infty) = Y(\infty)$
(i.e.\ independent of $\ell$), one of the biggest challenges is
obtaining bounds on the partial derivatives of $f_h(x)$ (usually up
  to third order). We refer to these as gradient bounds. In the
one-dimensional case, (\ref{eq:intropoisson}) is an ordinary
differential equation (ODE) that usually has a closed form
expression  that one can analyze directly, see for instance  \cite[Lemma 13.1]{ChenGoldShao2011}. However, when $d>1$ obtaining these gradient bounds
becomes significantly harder. By exploiting probabilistic solutions
to the Poisson equation,
gradient bounds were established for cases when $Y(\infty)$ is a
multivariate normal \cite{Barb1990}, multivariate Poisson
\cite{Barb1988} and multivariate Gamma \cite{Luk1994}.

\textbf{Component 2.} The next step is to produce the generator coupling. For that, we use
the basic adjoint relationship (BAR) for the stationary distribution
of $X^{(\ell)}(\infty)$. One can check that a random vector $X^{(\ell)}(\infty) \in \R^d$ has the
stationary distribution of the CTMC $X^{(\ell)}$ if and only if
\begin{equation}
  \label{eq:bar}
\E G_{X^{(\ell)}} f (X^{(\ell)}(\infty)) = 0
\end{equation}
for all functions $f:\R^d\to \R$ that have compact support.  For a given $h$, the
corresponding Poisson equation solution $f_h$ does not have compact
support. An important part of this step is to prove that
\eqref{eq:bar} continues to hold for $f_h$.  Thus, it follows from \eqref{eq:PoissonRep} and \eqref{eq:bar} that
\begin{equation} \label{eq:introgencoup}
\E h(X^{(\ell)}(\infty)) - \E h(Y^{(\ell)}(\infty)) = \E [G_{Y^{(\ell)}} f_h(X^{(\ell)}(\infty)) - G_{X^{(\ell)}} f_h (X^{(\ell)}(\infty))].
\end{equation}
Note that two random variables in the left side of
\eqref{eq:introgencoup} are typically defined on two different probability
spaces, whereas two random variables in the right side of
\eqref{eq:introgencoup} are all defined in terms of $X^{(\ell)}(\infty)$,
thus producing a coupling on a common probability space.


To bound the right side of  (\ref{eq:introgencoup}), we study
\begin{equation}
  \label{eq:generatorDiff}
  G_{X^{(\ell)}}f_h(x) -G_{Y^{(\ell)}}f_h(x)
\end{equation}
for each $x$ in the state space of $X^{(\ell)}$. By performing Taylor
expansion on $G_{X^{(\ell)}}f_h(x)$, we find that the difference 
involves the product of partial derivatives of $f_h$ and a term bounded by 
a polynomial of $x$.
Therefore, in addition to gradient bounds on $f_h$, in a lot of
cases we need bounds on various moments of
$\abs{X^{(\ell)}(\infty)}$  which we refer to as moment bounds. The main challenge is that both gradient and moment bounds must be
\textit{uniform} in $\ell$.

\textbf{Component 3.} In the last step, SSC comes into play when $X^{(\ell)}$ itself is not a
CTMC, but a projection of some higher dimensional CTMC $U^{(\ell)} = \{U^{(\ell)}(t)
\in \mathcal{U}, t \geq 0\}$, where the dimension of the state space
$\mathcal{U}$ is strictly greater than $d$. 
This is the case, for
example, in the $M/Ph/n+M$ system.
It is this difference in
dimensions that is responsible for most of the computational speedup in diffusion
approximations; most complex stochastic processing systems
exhibit some form of SSC
\cite{Reim1984b,BellWill2005,FoscSalz1978,Harr1998,HarrLope1999,
  Whit1971,Will1998a,Bram1998a,DaiTezc2011,EryiSrik2012}.
 Let $G_U$ be the generator of $U^{(\ell)}$ and
$U^{(\ell)}(\infty)$ have its stationary distribution. Now, BAR \eqref{eq:bar} becomes 
$G_{U^{(\ell)}} F(U^{(\ell)}(\infty))=0$ for each `nice' $F:\mathcal{U}\to \R$. Furthermore,
(\ref{eq:introgencoup}) becomes
\begin{equation} \label{eq:introgencoupSSC}
\E h(X^{(\ell)}(\infty)) - \E h(Y^{(\ell)}(\infty)) = \E [G_{Y^{(\ell)}} f_h(X^{(\ell)}(\infty)) - G_{U^{(\ell)}} F_h(U^{(\ell)}(\infty))],
\end{equation}
where $F_h: \mathcal{U}\to \R$ is the lifting of $f_h:\R^d\to \R$ defined by letting $x \in \R^d$ be the projection of $u \in \mathcal{U}$ and then setting 
\begin{equation}
  \label{eq:lifting}
  F_h(u) = f_h(x).
\end{equation}
As before, we can perform Taylor expansion on $G_{U^{(\ell)}}F_h(u)$ to simplify
the difference $G_{U^{(\ell)}}F_h(u)-G_{Y^{(\ell)}}f_h(x)$.  To use this difference to bound
the right side of (\ref{eq:introgencoupSSC}), we need a steady-state
SSC result for $U^{(\ell)}(\infty)$, which tells us how to approximate
$U^{(\ell)}(\infty)$ from $X^{(\ell)}(\infty)$ and guarantees that this
approximation error is small. To obtain our SSC result, we need to rely heavily on the
structure of the $M/Ph/n+M$ system. 

In \cite{Gurv2014}, Gurvich develops methodologies to 
prove statements similar to (\ref{eq:introresult}) for various
queueing systems. In particular, Gurvich develops important elements
of the first two components of our framework in the special case when
dim$(\mathcal{U}) = d$. Along the way, he independently rediscovers
many of the ideas central to Stein's method in the setting of
steady-state diffusion approximations.  He relies on the existence of
uniform Lyapunov functions for the diffusion processes. Putting the
Lyapunov functions together with the probabilistic solution for
(\ref{eq:intropoisson}) and a-priori Schauder estimates for elliptic
PDEs (see \cite{GilbTrud1983}), he is able to obtain uniform gradient
bounds for a large class of Poisson equations. Furthermore, he also
obtains the necessary uniform moment bounds using these Lyapunov
functions by showing that uniform moment bounds for the diffusion
process imply the same moments are uniformly bounded for the
CTMC. However, his result on uniform moment bounds no longer holds
when dim$(\mathcal{U}) > d$ due to the need for SSC, which poses an
additional technical challenge. We overcome this challenge for the
$M/Ph/n+M$ system in Lemma~\ref{lemma:CTMCmoments}, in which moment
bounds are established \emph{recursively}.

The work in \cite{Gurv2014} is conceptually close to this paper. In that paper, Gurvich packages all the components required to prove his results into several conditions, with the main condition being the existence of uniform Lyapunov functions for the diffusion processes. In contrast, a key contribution of our framework is its \textit{modular} nature. The immediate benefit we gain is the ability to apply this framework to cases when SSC occurs (dim$(\mathcal{U}) > d$). Moreover, although we also rely on Lyapunov functions to establish both moment and gradient bounds in our particular setting, our framework clearly illustrates that Lyapunov functions are merely tools one can use to establish these moment and gradient bounds; the bounds themselves are the actual drivers of our main results.

We have already mentioned that Lemma 13.1 of \cite{ChenGoldShao2011} presents a systematic way to establish gradient bounds in the one-dimensional setting ($d = 1$), and \cite{Barb1988, Barb1990, Luk1994} establish gradient bounds in the multi-dimensional setting ($d > 1$) for a few special cases of $Y^{(\ell)}(\infty)$. However, establishing multi-dimensional gradient bounds remains a very difficult problem that usually requires using structural properties of the distribution of $Y^{(\ell)}(\infty)$. Gurvich's use of a-priori Schauder estimates \cite{GilbTrud1983} together with Lyapunov functions represents the first systematic approach to establishing multi-dimensional gradient bounds. 

With regards to using Lyapunov functions to establish moment bounds, certain systems may not require moment bounds at all. For example, approximating the stationary distribution of the  simple birth-death process corresponding to a single-server queue does not require the use of moment bounds (although we do not consider the $M/M/1$ queue in this paper, Stein's method is easily applicable to it).  Thus, the modularity of our framework presents the components one needs to justify approximations for various systems, and promotes the view that Lyapunov functions are merely one of many tools to tackle the difficulties in these components.   

It is useful to compare the challenge level of each component in our framework. The generator coupling is the least challenging component, because the class of functions for which \eqref{eq:bar} holds is usually rich enough. The remaining major difficulties are moment bounds, gradient bounds and SSC. Moment bounds and SSC are a property of the CTMC sequence $\{X^{(\ell)}\}_{\ell=1}^{\infty}$, and the difficulty in establishing them will depend heavily on the CTMCs. On the other hand, gradient bounds are tied to the diffusion processes $\{Y^{(\ell)}\}_{\ell=1}^{\infty}$, and are typically only difficult to establish when the diffusion processes are multi-dimensional. One important class of multi-dimensional diffusion processes for which we do not have gradient bounds are semi-martingale reflected Brownian motions (SRBMs) \cite{HarrReim1981}. An SRBM can approximate networks of single-server queues, such as generalized Jackson networks. The Schauder gradient bounds of \cite{Gurv2014} are not immediately applicable to SRBMs, because the corresponding Poisson equation is defined on the non-negative orthant, and has oblique reflection boundary conditions.  



Stein's method is a powerful method that has been widely used
in probability, statistics, and their wide range of applications such
as bioinformatics; see, for example, the survey papers \cite{Ross2011,
  Chat2014}, the recent book \cite{ChenGoldShao2011} and the
references within. The connection between Stein's method and diffusion
processes was first made by Barbour in \cite{Barb1988, Barb1990}. In
the context of Stein's method, generator coupling is a realization of an
abstract concept that first appeared in the famous commutative diagram
in (28) of \cite{Stei1986}; a more refined explanation of which is
provided in (4) of \cite{Chat2014}. In particular, using Chatterjee's notation in
\cite{Chat2014}, our $\E G_{X^{(\ell)}} f_h (X^{(\ell)}(\infty))$ in \eqref{eq:introgencoup} is his $\E T \alpha
f(W)$. 

Diffusion approximations are usually ``justified'' by heavy traffic
limit theorems. 
It is proved in \cite{DaiHeTezc2010} that for our $M/Ph/n+M$ systems,
\begin{equation}
  \label{eq:processconvere}
  \tilde X^{(\lambda)}=\{\tilde X^{(\lambda)}(t), t\ge 0\} \Longrightarrow  Y=\{Y(t), t\ge
0\}
\end{equation}
as $\lambda$ goes to infinity while satisfying (\ref{eq:square-root}) (we use the arrival rate $\lambda$ to index these systems instead of the abstract $\ell$ as before).
Proving these limit theorems has been an active area of research in
the last 50 years; see, for example,
\cite{Boro1964,Boro1965,IgleWhit1970,IgleWhit1970a, Harr1978,Reim1984}
for single-class queueing networks, \cite{Pete1991, Bram1998a,
  Will1998a} for multiclass queueing networks,
\cite{KangKellLeeWill2009, YaoYe2012} for bandwidth sharing networks,
\cite{HalfWhit1981,Reed2009,DaiHeTezc2010} for many-server queues. The
convergence used in these limit theorems is the convergence in
distribution on the path space $\D([0, \infty), \R^d)$, endowed with
Skorohod $J_1$-topology \cite{EthiKurt1986,Whit2002}. The
$J_1$-topology on $\D([0,\infty), \R^d)$ essentially means convergence
in $\D([0, T], \R^d)$ for each $T>0$. In particular, it says nothing
about the convergence at ``$\infty$''. Therefore, these limit theorems
do not justify the steady-state convergence.


 In
\cite{DaiDiekGao2014}, the authors prove the convergence of
distribution $\tilde X^{(\lambda)}(\infty)$ to that of $Y(\infty)$ by proving an
interchange of limits.  The proof technique follows that of the
seminal paper \cite{GamaZeev2006}, where the authors prove an
interchange of limits for generalized Jackson networks of
single-server queues.  The results in \cite{GamaZeev2006} were
improved and extended by various authors for networks of
single-servers \cite{BudhLee2009, ZhanZwar2008, Kats2010}, for
bandwidth sharing networks~\cite{YaoYe2012}, and for many-server
systems \cite{Tezc2008, GamaStol2012, Gurv2014a}.  These ``interchange
limits theorems'' are qualitative and thus do not provide rates of
convergence as in (\ref{eq:introresult}).

\subsection{Notation}
\label{sec:notation}
All random variables and stochastic processes are defined on a common
probability space $(\Omega, \mathcal{F}, \mathbb{P})$ unless otherwise
specified. For a stochastic process $X = \{X(t), t \geq 0\}$ that has a unique stationary distribution we let
$X(\infty)$ be the random element having the stationary distribution
of $X$. For a sequence of random variables $\{X^n\}_{n=1}^{\infty}$, we write $X^n \Rightarrow X$ to denote convergence in distribution (also known as weak convergence) of $X^n$ to some random variable $X$. If $a > b$, we
adopt the convention that $\sum \limits_{i=a}^b (\cdot) = 0$. For an
integer $d \geq 1$, $\R^d$ denotes the $d$-dimensional Euclidean space
and $\Z_+^d$ denotes the space of $d$-dimensional vectors whose
elements are non-negative integers. For $a,b \in \R$, we define $a
\vee b = \max\{a,b\}$ and $a \wedge b = \min\{a,b\}$. For $x \in \R$,
we define $x^+ = x \vee 0$ and $x^- = (-x)\vee 0$.  For $x \in \R^d$,
we use $x_i$ to denote its $i$th entry and $\abs{x}$ to denote its
Euclidean norm. For $x,y \in \R^d$, we write $x \leq y$ when $x_i \leq
y_i$ for all $i$ and when $x \leq y$ we define the vector interval
$[x,y] = \{z: x \leq z \leq y\}$. All vectors are assumed to be
column vectors. We let $x^T$ and $A^T$ denote the transpose of a
vector $x$ and matrix $A$, respectively. For a matrix $A$, we use
$A_{ij}$ to denote the entry in the $i$th row and $j$th column. We reserve $I$
for the identity matrix, $e$ for the vector of all ones and $e^{(i)}$
for the vector that has a one in the $i$th element and zeroes
elsewhere; the dimensions of these vectors will  be clear from the context.
%

\subsection{Outline for Rest of Paper}
The rest of the paper is structured as
follows. Section~\ref{sec:model} formally defines the $M/Ph/n+M$
system as well as the diffusion process whose steady-state
distribution will approximate the system. Section~\ref{sec:mainresult}
states our main results. Section~\ref{sec:CTMCrep} describes the CTMC
representation of the $M/Ph/n+M$ system. Section~\ref{sec:mainproof}
introduces the first two components of our framework; the Poisson
equation and generator coupling. Section~\ref{sec:state-space-collapse} describes the SSC result, illustrating the third
component of our framework. It is here that the reader may see
the reason behind our slower rate of
convergence. This framework is then used in
Section~\ref{sec:thmpf} to prove our main
results. Appendix~\ref{app:proofs} contains the proofs for most of
the lemmas. 

\section{Models} \label{sec:model} 

In this section, we give additional description of the 
$M/Ph/n+M$ system and the corresponding diffusion model.

\subsection{The $M/Ph/n+M$ System}
The basic description of the $M/Ph/n+M$ queueing system was given
  in the first paragraph of the introduction. Here, we describe the
  dynamics of the system.  Upon arrival to the system with idle
servers, a customer begins service immediately. Otherwise, if all
servers are busy, the customer enters an infinite capacity queue to
wait for service. When a server completes serving a customer,
the server becomes idle if the queue is empty, or takes a customer
from the queue under the first-in-first-out (FIFO) service policy if
it is nonempty. Recall that the $Ph$ indicates that customer service times are
i.i.d.\ following a phase-type distribution. We shall provide a
definition of a phase-type distribution shortly below. The phase-type distribution
can approximate any positive-valued distribution \cite[Theorem
III.4.2]{Asmu2003}.

Recall that $\lambda$ denotes the arrival rate of the system.  We use
$1/\alpha$ to denote the mean patience time.  In our study, we take
the service time distribution and $\alpha$ fixed, but allow the
arrival rate $\lambda$ and  the number of servers
$n$ to grow without bound.
Throughout this paper,
 we assume that $n$ follows
the square-root-safety staffing rule in \eqref{eq:square-root}.
In
the pioneering paper of \cite{HalfWhit1981}, the authors studied these
systems as $\lambda\to\infty$ and $n$ grows to infinity following
\eqref{eq:square-root}. This parameter regime is now known as the
Halfin-Whitt regime. In this regime, the system has high server
utilization and at the same time has small customer waiting time and
abandonment fraction.  Therefore, this regime is also known as the
quality- and efficiency-driven (QED) regime, a term coined by
\cite{GansKoolMand2003}.

\subsubsection*{Phase-type Service Time  Distribution}\label{sec:phase-type-distr}
A phase-type distribution is assumed to have $d \geq 1$ phases. Each phase-type distribution is determined by the tuple $(p,\nu,P)$, where $p \in \R^d$ is a vector of non-negative entries whose sum is equal to one, $\nu \in \R^d$ is a vector of positive entries and $P$ is a $d \times d$ sub-stochastic matrix. We assume that $P$ is transient, i.e. 
\begin{equation}
(I-P)^{-1} \text{ \quad exists,} \label{eq:transience}
\end{equation}
and without loss of generality, we also assume that the diagonal entries of $P$ are zero ($P_{ii}=0$).

A random variable is said to have a phase-type distribution with parameters $(p,\nu,P)$ if it is equal to the absorption time of the following CTMC. The state space of the CTMC is $\{1, ... ,d+1\}$, with $d+1$ being the absorbing state. The CTMC starts off in one of the states in $\{1,...,d\}$ according to distribution $p$. For $i = 1, ... ,d $, the time spent in state $i$ is exponentially distributed with mean $1/\nu_i$. Upon leaving state $i$, the CTMC transitions to state $j= 1, ... ,d $ with probability $P_{ij}$, or gets absorbed into state $d+1$ with probability $1- \sum_{j=1}^d P_{ij}$.


The CTMC above is a useful way to describe the service times in the
$M/Ph/n+M$ system. Upon arrival to the system, a customer is assigned
her first service phase according to distribution $p$. If the customer is forced to wait in queue because all servers are busy, she is still assigned a first service phase, but this phase of service will not start until a server takes on this customer for service. Once a customer with initial phase $i$ enters service, her service time is the time until absorption to state $d+1$ by the CTMC. We
assume without loss of generality that for each service phase $i$,
either
\begin{equation} \label{eq:noredundancy}
p_i > 0 \text{ or } P_{ji}> 0 \text{ for some $j$}.
\end{equation}
This simply means that there are no redundant phases. 

We now define some useful quantities for future use. Define 
\begin{equation}
R = (I-P^T)\text{diag}(\nu) \quad  \text{and} \quad \gamma = \mu R^{-1}p, \label{eq:defph}
\end{equation}
where the matrix $\text{diag}(\nu)$ is the $d\times d$ diagonal matrix 
with diagonal entries given by the components of $\nu$. One may verify that $\sum \limits_{i=1}^d \gamma_i = 1$. One can interpret $\gamma_i$ to be the fraction of phase $i$ service load on the $n$ servers.

For concreteness, we provide two examples of phase-type distributions when $d=2$. The first example is the two-phase hyper-exponential distribution, denoted by $H_2$. The corresponding tuple of parameters is $(p,\nu,P)$, where
\begin{displaymath}
p = (p_1, p_2)^T, \quad \nu = (\nu_1,\nu_2)^T, \text{ \quad and \quad } P = 0.
\end{displaymath}
Therefore, with probability $p_i$, the service time follows an exponential distribution with mean $1/\nu_i$.

The second example is the Erlang-$2$ distribution, denoted by $E_2$. The corresponding tuple of parameters is $(p,\nu,P)$, where
\begin{displaymath}
p = (1,0)^T, \quad \nu = (\theta, \theta)^T, \text{ \quad and \quad } P = \begin{pmatrix}
0 & 1\\ 
0 & 0
\end{pmatrix}.
\end{displaymath}
An $E_2$ random variable is a sum of two i.i.d.\ exponential random variables, each having mean $1/\theta$.

\subsection{System Size Process and Diffusion Model}
\label{sec:diffusion}

Before we state the main results, we introduce the process we wish to approximate, as well as the approximating diffusion process -- the piecewise OU process.
Recall that  $X = \{X(t) \in \R^d, t \geq 0\}$ is the system size process, where 
\begin{displaymath}
X(t) = (X_1(t), ... , X_d(t))^T,
\end{displaymath}
and $X_i(t)$ is the number of customers of phase $i$ in the system (queue + service) at time $t$. We emphasize that $X$ is not a CTMC, but it is a deterministic function of a higher-dimensional CTMC, which will be described in Section \ref{sec:CTMCrep}. 

The process $X$ depends on $\lambda, n, \alpha, p, P$, and $\nu$. However, in this paper we keep $\alpha, p, P$, and $\nu$ fixed, and allow $\lambda$ and $n$ to vary according to \eqref{eq:square-root}. For the remainder of the paper we write $X^{(\lambda)}$ to emphasize the dependence of $X$ on $\lambda$; the dependence of $X^{(\lambda)}$ on $n$ is implicit through \eqref{eq:square-root}.

Recall the definition of $\gamma$ in (\ref{eq:defph}) and define the
scaled random variable
\begin{equation}
  \label{eq:scaledctmcsd}
\tilde X^{(\lambda)}(\infty) = \delta (X^{(\lambda)}(\infty) - \gamma n),
\end{equation}
where, for convenience, we let 
\begin{equation}
  \label{eq:delta}
\delta = 1/\sqrt{\lambda}.  
\end{equation}
To approximate $\tilde X^{(\lambda)}(\infty)$, we introduce the piecewise OU process $Y =
\{Y(t), t\geq 0\}$. This is a $d$-dimensional diffusion process satisfying
\begin{equation}
Y(t) = Y(0) - p \beta t - R \int_0^t{\big(Y(s) -p(e^{T}Y(s))^+\big) ds} - \alpha p \int_0^t {(e^{T}Y(s))^+ds} + \sqrt{\Sigma} B(t). \label{eq:defou}
\end{equation}
Above, $B(t)$ is the $d$-dimensional standard Brownian motion and $\sqrt{\Sigma}$ is any $d\times d$ matrix satisfying
\begin{equation} \label{eq:OUdiffusioncoeff}
\sqrt{\Sigma}\sqrt{\Sigma}^T= \Sigma = \text{diag}(p) + \sum_{k=1}^d \gamma_k \nu_k H^k + (I-P^T)\text{diag}(\nu) \text{diag}(\gamma)(I-P),
\end{equation}
where the matrix $H^k$ is defined as 
\begin{displaymath}
H^k_{ii} = P_{ki}(1-P_{ki}), \quad H^k_{ij} = -P_{ki}P_{kj} \quad \text{ for $j\neq i$}.
\end{displaymath} 
Comparing the form of $\Sigma$ above to (2.24) of \cite{DaiHe2013} confirms that it is positive definite. Thus $\sqrt{\Sigma}$ exists. Observe that $Y$ depends only on $\beta, \alpha, p, P$, and $\nu$, all of which are held constant throughout this paper.


The diffusion process in (\ref{eq:defou}) has been studied by \cite{DiekGao2013}. They prove that $Y$ is positive recurrent by finding an appropriate Lyapunov function. In particular, this means that $Y$ admits a stationary distribution.

\section{Main Results} \label{sec:mainresult}

We now state our main results.

\begin{theorem} \label{thm:main}
For every integer $m > 0$, there exists a constant $C_m = C_m(\beta,\alpha,p,\nu,P)>0$ such that for all locally Lipschitz functions $h : \R^d \to \R$ satisfying 
\begin{displaymath}
\abs{h(x)} \leq \abs{x}^{2m} \quad \text{  for  } x\in \R^d,
\end{displaymath}
we have
\begin{displaymath} \label{eq:thm1}
\abs{ \E h(\tilde X^{(\lambda)}(\infty)) - \E h(Y(\infty))} \leq \frac{C_m}{\sqrt{\lambda}}
\quad \text{for all } \lambda >0
\end{displaymath} 
satisfying \eqref{eq:square-root}, which we recall below as
\begin{displaymath}
 n \mu = \lambda + \beta \sqrt{\lambda}.
\end{displaymath}
\end{theorem}
Theorem~\ref{thm:main} will be proved in  Section~\ref{sec:thmpf}. As a consequence of the theorem, we immediately have the following corollary.
\begin{corollary} \label{corol:main}
There exists a constant $C_1 = C_1(\beta,\alpha,p,\nu,P)>0$ such that

\begin{displaymath}
  \sup \limits_{h \in \mathcal{W}^{(d)}}  \abs{\E h(\tilde X^{(\lambda)}(\infty)) - \E h(Y(\infty))} \leq \frac{C_1}{\sqrt{\lambda}} \quad \text{for all } \lambda >0 
\end{displaymath}
satisfying \eqref{eq:square-root}, where $W^{(d)}$ is defined in \eqref{eq:Wd}.
In particular, 
\begin{displaymath}
\tilde X^{(\lambda)}(\infty) \Rightarrow Y(\infty) \text{ \quad as  \quad } \lambda \rightarrow \infty. 
\end{displaymath}
\end{corollary}
\begin{proof}
Suppose $h \in \mathcal{W}^{(d)}$. Without loss of generality, we may assume that $h(0) = 0$, otherwise we may simply consider $h(x) - h(0)$. By definition of $\mathcal{W}^{(d)}$, 
\begin{displaymath}
\abs{h(x)} \leq \abs{x} \quad \text{ for } x\in \R^d
\end{displaymath}
and the result follows \text{from Theorem~\ref{eq:thm1} with $m=1$}.
\end{proof}

\begin{remark}\label{rem:lambdarange}
  For any fixed $\beta\in \R$, there are only finitely many
  combinations of $\lambda \in (0, 4)$ and integer $n\ge 1$ satisfying
  \eqref{eq:square-root}. Therefore, it suffices to prove Theorem
  \ref{thm:main} by restricting $\lambda\ge 4$, a convenience for technical purposes.
\end{remark}

\section{Markov Representation} \label{sec:CTMCrep}
The $M/Ph/n+M$ system can be represented as a CTMC
\begin{displaymath}
U^{(\lambda)} = \{ U^{(\lambda)}(t), t \geq 0\}
\end{displaymath} taking values in $\mathcal{U}$, the set of finite sequences $\{u_1, ... , u_k\}$ . The sequence $u=\{u_1, ... , u_k\}$ encodes the service phase of each customer and their order of arrival to the system. For example, the sequence $\{5, 1, 4\}$ corresponds to $3$ customers in the system, with the service phases of the first, second and third customers (in the order of their arrival to the system) being $5$, $1$ and $4$, respectively. We use $\abs{u}$ to denote the length of the sequence $u$. The irreducibility of the CTMC $U^{(\lambda)}$ is guaranteed by
(\ref{eq:transience}) and (\ref{eq:noredundancy}).

 We remark here that $U^{(\lambda)}$ is not the simplest Markovian representation of the $M/Ph/n+M$ system. Another way to represent this system would be to consider a $d+1$ dimensional CTMC that keeps track of the total number of customers in the system, as well as the total number of customers in each phase that are currently in service; this $d+1$ dimensional CTMC is used in \cite{DaiHeTezc2010}. In this paper we use the infinite dimensional CTMC $U^{(\lambda)}$ because the system size process $X^{(\lambda)}$ cannot be recovered sample path wise from the $d+1$ dimensional CTMC, it can only be recovered from $U^{(\lambda)}$. Also, the CTMC $U^{(\lambda)}$ will play an important role in our SSC argument in Section~\ref{sec:state-space-collapse}.

In addition to the system size process $X^{(\lambda)}$, we define the queue size process $Q^{(\lambda)} = \{Q^{(\lambda)}(t) \in \Z^d_+, t \geq 0\}$, where 
\begin{displaymath}
Q^{(\lambda)}(t) = (Q^{(\lambda)}_1(t), ... , Q^{(\lambda)}_d(t))^T,
\end{displaymath}
and $Q^{(\lambda)}_i(t)$ is the number of customers of phase $i$ in the queue at time $t$.
Then $X^{(\lambda)}_i(t) - Q^{(\lambda)}_i(t) \geq 0$ is the number phase $i$ customers in service at time $t$.

To recover $X^{(\lambda)}(t)$ and $Q^{(\lambda)}(t)$ from $U^{(\lambda)}(t)$, we define the projection functions $\Pi_X : \mathcal{U} \to \R^d$  and $\Pi_Q : \mathcal{U} \to \R^d$. For each $u\in {\cal U}$ and each phase $i\in \{1, \ldots, d\}$,
\begin{displaymath}
\left(\Pi_X(u)\right)_{i} = \sum_{k=1}^{\abs{u}} 1_{\{ u_k = i\}}
\quad \text{ and }
\quad 
\left(\Pi_Q(u)\right)_{i} = \sum_{k = n+1}^{\abs{u}} 1_{\{ u_k = i\}}.
\end{displaymath}
It is clear that on each sample path
\begin{equation} \label{eq:projectionCTMC}
X^{(\lambda)}(t) = \Pi_X(U^{(\lambda)}(t)) \text{\quad and \quad} Q^{(\lambda)}(t) = \Pi_Q(U^{(\lambda)}(t)) \quad \text{ for } t\ge 0.
\end{equation} 
Because there is customer abandonment the Markov chain $U^{(\lambda)}$ can be
proved to be positive recurrent with a unique stationary
distribution \cite{DaiDiekGao2014}. We use $U^{(\lambda)}(\infty)$ to denote the random element that has
the stationary distribution. It follows that $X^{(\lambda)}(\infty)=\Pi_X(U^{(\lambda)}(\infty))$ has the stationary distribution of $X^{(\lambda)}$, and  $\tilde X^{(\lambda)}(\infty)$ in \eqref{eq:scaledctmcsd} is given by
\begin{equation}
  \label{eq:Xinfty}
  \tilde X^{(\lambda)}(\infty) = \delta (\Pi_X(U^{(\lambda)}(\infty))-\gamma n).
\end{equation}

For $u \in \mathcal{U}$, we define 
\begin{equation}\label{eq:projlittle}
x = \delta(\Pi_X(u) - \gamma n), \quad  q = \Pi_Q(u) \text{ \quad and \quad} z = \Pi_X(u) - q.
\end{equation}
When the CTMC is in state $u$, we interpret $(\Pi_X(u))_i$, $q_i$, and $z_i$ as the number of the
  phase $i$ customers in system, in queue, and in service,
  respectively. It follows that $z\ge 0$. 
  
Let $G_{U^{(\lambda)}}$ be the generator of the CTMC $U^{(\lambda)}$. To describe it, we introduce the lifting operator $A$. For any function $f: \R^d \to \R$, we define $Af: \mathcal{U} \to \R$ by 
\begin{equation} \label{eq:lifter}
Af(u) = f(\delta (\Pi_X(u)-\gamma n)) = f(x).
\end{equation}
Hence, for any function $f: \R^d \to \R$, the generator acts on the lifted version $Af$ as follows:
\begin{eqnarray}
G_{U^{(\lambda)}} Af(u) &=& \sum \limits_{i=1}^d \lambda p_i( f(x + \delta e^{(i)}) - f(x)) + \sum \limits_{i=1}^d \alpha q_i (f(x - \delta e^{(i)}) - f(x)) \nonumber\\
&& + \sum \limits_{i=1}^d \nu_i z_i \Big{[} \sum \limits_{j=1}^d P_{ij}f(x+\delta e^{(j)}-\delta e^{(i)}) \nonumber\\
&& {} + (1-\sum \limits_{j=1}^d P_{ij})f(x-\delta e^{(i)}) -f(x)  \Big{]}.
\label{eq:tildeG}
\end{eqnarray}

Observe that $G_{U^{(\lambda)}} Af(u)$ does not depend on the entire sequence $u$;
it depends on $x$, $q$, and the function $f$ only.

\section{The Generator Coupling of Stein's  Method} \label{sec:mainproof} This section is devoted to developing
a generator coupling of Stein's method.  This framework will be used
in Section \ref{sec:thmpf} to prove Theorem~\ref{thm:main}.
\subsection{Poisson Equation}
The main idea behind Stein's method is that instead of bounding 
\begin{equation}
  \label{eq:diff}
\E h(\tilde X^{(\lambda)}(\infty)) - \E h(Y(\infty)),  
\end{equation}
one solves the Poisson equation 
\begin{equation} \label{eq:poisson}
G_Y f_h(x) = h(x) - \E h(Y(\infty)),
\end{equation}
where the generator $G_Y$ of the diffusion process $Y$, applied to a function $f\in C^2(\R^d)$, is given by 
\begin{eqnarray}
&G_Y f(x) = \sum \limits_{i=1}^d \partial_i f(x) \Big{[} p_i \beta - \nu_i(x_i - p_i(e^Tx)^+) - \alpha p_i (e^Tx)^+ + \sum \limits_{j=1}^d P_{ji}\nu_j(x_j-p_j(e^Tx)^+)\Big{]} \notag \\
&+ \frac{1}{2}\sum \limits_{i,j=1}^d \Sigma _{ij} \partial_{ij} f(x)
\quad \text{ for } x\in \R^d. \label{eq:DMgen}
\end{eqnarray}
Then, to bound the difference in (\ref{eq:diff}), it is sufficient to 
find a bound on
\begin{equation}\label{eq:poissonLHS}
\E G_Y f_h(\tilde X^{(\lambda)}(\infty)).
\end{equation}

The following lemma, based on the results of \cite{Gurv2014},  guarantees the existence of a solution to (\ref{eq:poisson}) and provides gradient bounds for it. The proof of this lemma is given in Section \ref{sec:proof-lemma-refl}.

\begin{lemma} \label{lemma:poisson}
For any locally Lipschitz function $h: \R^d \to \R$ satisfying $\abs{h(x)} \leq \abs{x}^{2m}$, equation (\ref{eq:poisson})
has a solution $f_h$. Moreover, there exists a constant $C(m,1)>0$ (depending only on $(\beta,\alpha,p,\nu,P)$) such that  for $x\in \R^d$
\begin{eqnarray}
\abs{f_h(x)} &\leq & C(m,1)(1+\abs{x}^2)^m, \label{eq:gradbound1}\\
\abs{\partial_i f_h(x)} &\leq & C(m,1)(1+\abs{x}^2)^m(1+\abs{x}), \label{eq:gradbound2} \\
\abs{\partial_{ij} f_h(x)} &\leq & C(m,1)(1+\abs{x}^2)^m(1+\abs{x})^2, \label{eq:gradbound3} \\
\sup \limits_{y\in \R^d:\abs{y-x} < 1} \frac{\abs{\partial_{ij} f_h(y)-\partial_{ij}f_h(x)}}{\abs{y-x}} &\leq &C(m,1) (1+\abs{x}^2)^m(1+\abs{x})^3. \label{eq:gradbound4}
\end{eqnarray}
\end{lemma}

\subsection{Generator Coupling}
Let $W^{(\lambda)}$ denote the random variable $G_Y f_h(\tilde X^{(\lambda)}(\infty))$ in (\ref{eq:poissonLHS}). To
prove $\abs{\E W^{(\lambda)}}$ small, a common approach in using the Stein's method is to
find a coupling $\tilde W^{(\lambda)}$ for $W^{(\lambda)}$ so that 
\begin{eqnarray}
&&\text{$\abs{\E\tilde W^{(\lambda)}}$ is
small, and }  \hspace{.6\textwidth}\label{item:1}   \\
&&  \text{$\E\abs{W^{(\lambda)}-\tilde W^{(\lambda)}}$ is small.} \label{item:2}
\end{eqnarray}
Constructing an
effective coupling is an art that is problem specific. See \cite{Ross2011}
for a recent survey that includes examples of various couplings.

We use $\tilde W^{(\lambda)} = G_{U^{(\lambda)}} Af_h(U^{(\lambda)}(\infty))$ to construct
the coupling, where $A$ is the lifting operator defined in (\ref{eq:lifter}). The following lemma justifies the coupling propety (\ref{item:1}).
\begin{lemma} \label{lemma:CTMCbar}
Let $h: \R^d \to \R$ satisfy $\abs{h(x)} \leq \abs{x}^{2m}$. The function $f_h$ given by (\ref{eq:poisson}) satisfies
\begin{equation} \label{eq:CTMCbar}
\E G_{U^{(\lambda)}} Af_h(U^{(\lambda)}(\infty)) = 0.
\end{equation}
\end{lemma}
To prove the lemma, we need finite moments of the steady-state system size.

\begin{lemma}\label{lem:finteexpmoment}
(a) Let  $L(u) =\exp(e^T \Pi_X(u))$ for $u \in \mathcal{U}$. Then 
\begin{equation} \label{eq:CTMCmgfexist}
\E L(U^{(\lambda)}(\infty))<\infty.
\end{equation}
(b) all moments of $e^TX^{(\lambda)}(\infty)$ are finite.
\end{lemma}
\begin{proof}
One may verify that
\begin{displaymath}
G_{U^{(\lambda)}} L(u) \leq \lambda(\exp(1)-1)L(u) - \alpha (e^T \Pi_X(u) -n)^+(1- \exp(-1))L(u).
\end{displaymath}
It follows that there exist a positive constant $C= C(\lambda, n, \alpha)$ such that,  whenever $e^T \Pi_X(u)$ is large enough,
\begin{equation} \label{eq:CTMCfosterlyap}
G_{U^{(\lambda)}} L(u) \leq -CL(u) + 1.
\end{equation}
Part (a) follows from \cite[Theorem 4.2]{MeynTwee1993b}. Part (b)
follows from \eqref{eq:CTMCmgfexist} and the equality $e^T
\Pi_X(U^{(\lambda)}(\infty)) = e^T X^{(\lambda)}(\infty)$.
\end{proof}
The function $L(u)$ is said to be a Lyapunov function.
Inequality (\ref{eq:CTMCfosterlyap}) is known as a Foster-Lyapunov
condition and guarantees that the CTMC is positive recurrent; see, for example, 
\cite{MeynTwee1993b}.

\begin{proof}[Proof of Lemma~\ref{lemma:CTMCbar}]
A sufficient condition for (\ref{eq:CTMCbar}) to hold is given by \cite[Proposition 1.1]{Hend1997} (alternatively, see \cite[Proposition 3]{GlynZeev2008}), namely
\begin{equation}
  \label{eq:glynnzeevi}
  \E\Big[ \abs{G_{U^{(\lambda)}}(U^{(\lambda)}(\infty),U^{(\lambda)}(\infty))} \abs{Af_h(U^{(\lambda)}(\infty))}\Big]<\infty.
\end{equation}
Above, $G_{U^{(\lambda)}}(u,u)$ is the $u$th diagonal entry of the generator matrix $G_{U^{(\lambda)}}$.
 In our case, the left side of (\ref{eq:glynnzeevi}) is equal to 
\begin{eqnarray*}
& =& \E\Big[ \abs{G_{U^{(\lambda)}}(U^{(\lambda)}(\infty),U^{(\lambda)}(\infty))} \abs{f_h(\tilde X^{(\lambda)}(\infty))}\Big]\\
&=&\E \abs{\lambda + \alpha(e^TX^{(\lambda)}(\infty) - n)^+ + \sum \limits_{i=1}^d \nu_i (X^{(\lambda)}_i(\infty) - Q^{(\lambda)}_i(\infty)} \abs{ f_h(\tilde X^{(\lambda)}(\infty))} \notag \\
&\leq & \E \abs{\lambda + (\alpha \vee \max_i \{\nu_i\}) e^TX^{(\lambda)}(\infty)} \abs{ f_h(\tilde X^{(\lambda)}(\infty))},
\end{eqnarray*}
where the first equality follows from \eqref{eq:Xinfty} and \eqref{eq:lifter}.
One may apply (\ref{eq:gradbound1}) and (\ref{eq:CTMCmgfexist}) to see that the quantity above is finite.
\end{proof}
\subsection{Taylor Expansion}

To prove that the coupling $\tilde W^{(\lambda)} = G_{U^{(\lambda)}} Af_h(U^{(\lambda)}(\infty))$ satisfies the coupling property (\ref{item:2}), we need to prove 
that
\begin{equation*}
  \label{eq:compareGener}
\E\abs{W^{(\lambda)}-\tilde W^{(\lambda)}}  = \E \abs{G_{U^{(\lambda)}} Af_h(U^{(\lambda)}(\infty)) - G_Y f_h(\tilde X^{(\lambda)}(\infty))}
\end{equation*}
is small.  For that, we compare the generator $G_{U^{(\lambda)}}$ of the
CTMC with $G_Y$.
By performing Taylor expansion on $G_{U^{(\lambda)}} Af_h(u)$ in (\ref{eq:tildeG}), one has 
\begin{eqnarray}
G_{U^{(\lambda)}} Af_h(u) &=& \sum \limits_{i=1}^d \lambda p_i\big( \delta \partial_i f_h(x) + \frac{\delta^2}{2} \partial_{ii}f_h(\xi_{i}^+)\big) +\alpha q_i\big(-\delta \partial_i f_h(x)+  \frac{\delta^2}{2}\partial_{ii}f_h(\xi_{i}^-)\big)\nonumber\\
&& + \sum \limits_{i=1}^d \nu_i z_i\Big{[}(1-\sum \limits_{j=1}^d P_{ij})\big(-\delta \partial_i f_h(x)+ \frac{\delta^2}{2}\partial_{ii}f_h(\xi_{i}^-)\big)
+\sum \limits_{j=1}^d P_{ij}\Big(-\delta \partial_i f_h(x) 
\nonumber \\
&& \qquad  + \delta \partial_j f_h(x) 
+ \frac{\delta^2}{2}\partial_{ii}f_h(\xi_{ij}) + \frac{\delta^2}{2}\partial_{jj}f_h(\xi_{ij})- \delta^2 \partial_{ij}f_h(\xi_{ij})\Big)\Big{]}, \label{eq:GXtaylor}
\end{eqnarray}
where $\xi_{i}^+ \in [x, x+\delta e^{(i)}]$, $\xi_{i}^-\in
[x-\delta e^{(i)}, x] $ and $\xi_{ij}$ lies somewhere between $x$ and
$x-\delta e^{(i)}+\delta e^{(j)}$.  Using the gradient bounds in Lemma~\ref{lemma:poisson}, we have the following lemma, which will be proved 
in Section \ref{app:taylorexp}.

\begin{lemma}\label{lemma:diffbound} There exists a constant
  $C(m,2)>0$ (depending only on $(\beta,\alpha,p,\nu,P)$) such that for any $u \in \mathcal{U}$,
  \begin{eqnarray}
    \label{eq:diffbound}
&&G_{U^{(\lambda)}} Af_h(u) - G_Y f_h(x) \nonumber\\
&=&  \sum \limits_{i=1}^d \partial_i f_h(x)\Big{[}(\nu_i - \alpha-\sum \limits_{j=1}^d P_{ji}\nu_j)( \delta q_i - p_i(e^T x)^+)\Big{]} + E(u),
  \end{eqnarray}
  where $q$ and $x$ are as in (\ref{eq:projlittle}), $\delta$ as in \eqref{eq:delta}, and $E(u)$ is an error term that satisfies 

\begin{displaymath}
\abs{E(u)} \leq \delta\, C(m,2) (1+\abs{x}^2)^m (1+\abs{x})^4.
\end{displaymath}
\end{lemma}

\section{State Space Collapse}\label{sec:state-space-collapse}

One of the challenges we face comes from the fact that our CTMC $U^{(\lambda)}$ is infinite-dimensional, while the approximating diffusion process is only $d$-dimensional. 
Recall the process $(X^{(\lambda)},Q^{(\lambda)})$ defined in (\ref{eq:projectionCTMC}) and the lifting operator $A$ acting on functions $f:\R^d \to \R$, as defined in (\ref{eq:lifter}). When acting on the lifted functions $Af(U^{(\lambda)}(\infty))$, the CTMC generator $G_{U^{(\lambda)}}$ depends on both $\tilde X^{(\lambda)}(\infty)$ and $Q^{(\lambda)}(\infty)$, but its approximation $G_Y f(\tilde X^{(\lambda)}(\infty))$ only depends on $\tilde X^{(\lambda)}(\infty)$. This is captured in \eqref{eq:diffbound} by the term
\begin{displaymath}
\sum \limits_{i=1}^d \partial_i f_h(x)\Big{[}(\nu_i - \alpha-\sum \limits_{j=1}^d P_{ji}\nu_j)( \delta q_i - p_i(e^T x)^+)\Big{]}.
\end{displaymath}
To bound this term, observe that for any $1 \leq i \leq d$, 
\begin{eqnarray}
&&\Big(\nu_i - \alpha-\sum \limits_{j=1}^d P_{ji}\nu_j\Big)\partial_i f_h(x)\big( \delta q_i - p_i(e^T x)^+\big) \notag \\
 &=& \Big(\nu_i - \alpha-\sum \limits_{j=1}^d P_{ji}\nu_j\Big)\Big(\partial_i f_h(x)- \partial_i f_h\big(x - \delta q + p(e^T x)^+\big) \Big)\big( \delta q_i - p_i(e^T x)^+\big) \notag  \\
 &&+\ \Big(\nu_i - \alpha-\sum \limits_{j=1}^d P_{ji}\nu_j\Big)\partial_i f_h\big(x - \delta q + p(e^T x)^+\big)\big( \delta q_i - p_i(e^T x)^+\big) \notag \\
 &=& \Big(\nu_i - \alpha-\sum \limits_{j=1}^d P_{ji}\nu_j\Big)\sum_{k=1}^{d}\partial_{ik} f_h(\xi)( \delta q_k - p_k(e^T x)^+)\big( \delta q_i - p_i(e^T x)^+\big)  \notag \\
 &&+\ \Big(\nu_i - \alpha-\sum \limits_{j=1}^d P_{ji}\nu_j\Big)\partial_i f_h\big(\delta (z - \gamma n) + p(e^T x)^+\big)\big( \delta q_i - p_i(e^T x)^+\big), \label{eq:interm_ssc}
\end{eqnarray}
where $z$, defined in \eqref{eq:projlittle}, is a vector that represents the number of customers of each type in service, and $\xi$ is some point between $x$ and $x - \delta q + p(e^T x)^+$. In particular, there exists some constant $C$ that doesn't depend on $\lambda$ and $n$, such that
\begin{equation} \label{eq:xi_ineq}
\abs{\xi} \leq \abs{x}  + \delta \abs{q} + \abs{p}(e^T x)^+ \leq C \abs{x},
\end{equation}
because $\delta q_i \leq (e^T x)^+$ for each $1 \leq i \leq d$ (i.e.\ the number of phase $i$ customers in queue can never exceed the queue size).

In order to bound the expected value of \eqref{eq:interm_ssc}, we must prove a
relationship between $\tilde X^{(\lambda)}(\infty)$ and $Q^{(\lambda)}(\infty)$. Intuitively, the number of customers of phase $i$
waiting in the queue should be approximately equal to a fraction $p_i$
of the total queue size. The following two lemmas bound
the error caused by the SSC approximation. They are proved at the end of this section.

\begin{lemma}
\label{lemma:SSC}
Let $Z^{(\lambda)}(\infty) = X^{(\lambda)}(\infty) - Q^{(\lambda)}(\infty)$  be the vector representing the number of customers of each type in service in steady-state. Then conditioned on $(e^T \tilde X^{(\lambda)}(\infty))^+$, the random vectors $Q^{(\lambda)}(\infty)$ and $Z^{(\lambda)}(\infty)$ are independent.
Furthermore,
\begin{equation} \label{eq:multinom}
\E \Big [\delta Q^{(\lambda)}(\infty) - p(e^T \tilde X^{(\lambda)}(\infty))^+ \Big|\ (e^T \tilde X^{(\lambda)}(\infty))^+ \Big] = 0,
\end{equation}
and for any integer $m>0$, there exists $C(m, 3)>0$ (depending only on $(\beta,\alpha,p,\nu,P)$) such that for all $\lambda>0$ and $n\ge 1$ satisfying \eqref{eq:square-root},
\begin{equation}
  \label{eq:sscscaled}
  \E \Big [\abs{\delta Q^{(\lambda)}(\infty) - p(e^T \tilde X^{(\lambda)}(\infty))^+}^{2m}\Big ] \leq \delta^m\, C(m, 3)\E [(e^T \tilde X^{(\lambda)}(\infty))^+]^m,
\end{equation}
where $\delta=1/\sqrt{\lambda}$ as in \eqref{eq:delta}. 
\end{lemma}

\begin{lemma}
\label{lemma:SSC_bound}
For any integer $m>0$, there exists $C(m, 4)>0$ (depending only on $(\beta,\alpha,p,\nu,P)$) such that for any locally Lipschitz function $h: \R^d \to \R$ satisfying $\abs{h(x)} \leq \abs{x}^{2m}$, and all $\lambda>0$ and $n\ge 1$ satisfying \eqref{eq:square-root}
\begin{eqnarray}
&&\abs{\sum \limits_{i=1}^d  \E \bigg[\partial_i f_h(\tilde X^{(\lambda)}(\infty))\Big{[}(\nu_i - \alpha-\sum \limits_{j=1}^d P_{ji}\nu_j)( \delta Q^{(\lambda)}_i(\infty) - p_i(e^T \tilde X^{(\lambda)}(\infty))^+)\Big{]}\bigg]} \notag \\
&\leq & \delta C(m, 4)\E \Big[\big((e^T \tilde X^{(\lambda)}(\infty))^+\big)^2\Big]\sqrt{\E \Big[1 + \abs{\tilde X^{(\lambda)}(\infty)}^8\Big]}
\end{eqnarray}
where $f_h(x)$ is the solution to the Poisson equation \eqref{eq:poisson}.
\end{lemma}
\begin{proof}[Proof of Lemma~\ref{lemma:SSC}]
We begin by proving \eqref{eq:sscscaled}, for which it suffices to show that for all $\lambda>0$ and $n\ge 1$
    satisfying \eqref{eq:square-root}
\begin{displaymath}
\E \Big [\abs{Q^{(\lambda)}(\infty) - p(e^T X^{(\lambda)}(\infty) - n)^+}^{2m}\Big ] \leq C(m, 3)\E [(e^TX^{(\lambda)}(\infty) - n)^+]^m.
\end{displaymath}

  We first prove a version of \eqref{eq:sscscaled} for any finite time $t \geq 0$. Then,
  $(e^TX^{(\lambda)}(t) - n)^+$ is the total number of customers waiting in queue
  at time $t$. Assume that the system is empty at time $t = 0$, i.e.\ $X^{(\lambda)}(0) = 0$. Fix a phase $i$.  Upon arrival to the system, a
  customer is assigned to service phase $i$ with probability $p_i$. Consider the sequence $\{\xi_j:j=1, 2, \ldots\}$, where $\xi_j$ is one if the $j$th customer to enter the system was assigned to phase $i$, and zero otherwise. Then $\{\xi_j:j=1, 2, \ldots\}$ is a sequence of iid Bernoulli random
  variables with $\Prob(\xi_j=1)=p_i$. For $t > 0$, define $A(t)$ and $B(t)$ to be the total number of customers to have entered the system, and entered service by time $t$, respectively. Also let $\zeta_j(t)$ be the indicator of whether customer $j$ is still waiting in queue at time $t$. Then 
  \begin{eqnarray}
    \label{eq:rigL}
 &&    (e^TX^{(\lambda)}(t) - n)^+ = \sum_{j=B(t)+1}^{A(t)} \zeta_j(t), \\
 &&   Q^{(\lambda)}_i(t) = \sum_{j= B(t) + 1}^{A(t)} \xi_j \zeta_j(t).
    \label{eq:rigQ}
  \end{eqnarray}
Let $Z^{(\lambda)}(t) = X^{(\lambda)}(t) - Q^{(\lambda)}(t)$ be the vector keeping track of the customer types in service at time $t$ and let $B(\ell,p_i)$ be a binomial random variable with $\ell\in \Z_+$ trials and success probability $p_i$. Assuming $X^{(\lambda)}(0)=0$, by a sample path construction of the process $U^{(\lambda)}$ one can verify that for any time $t \geq 0$, the following three properties hold.
First, for any $z \in \Z_+^d$, $a,b \in \Z_+$ \text{with $a\ge 1$}, and $x_1, \ldots, x_a, y_1, \ldots, y_a \in \{0,1\}$,
\begin{eqnarray}
&&
\Prob\big(\xi_{b+1} = x_1, \ldots, \xi_{b+a}=x_{a}\ |\ A(t)= b+a, B(t)=b, Z^{(\lambda)}(t) = z, \notag\\
&& {} \hspace{2in} \zeta_{b+1}=y_1, \ldots, \zeta_{b+a}=y_{a} \big) \notag \\
& =&\ \Prob\big(\xi_{1} = x_1\big)\Prob\big(\xi_{2} = x_2\big)\ldots \Prob\big(\xi_{a}=x_{a} \big)\notag \\
& =& \ p_i^{\sum_{i=1}^a x_i}(1- p_i)^{a-\sum_{i=1}^a x_i}. \label{eq:rig1}
\end{eqnarray}
The right side of \eqref{eq:rig1} is independent of $b$, $z$, $y_1,
\ldots, y_a$.  It then follows from \eqref{eq:rigL}, \eqref{eq:rigQ} 
and \eqref{eq:rig1} that  for any integer $\ell\ge 1$, $q_i\in \Z_+$, and $z \in \Z_+^d$,
\begin{align}
&\Prob\big(Q^{(\lambda)}_i(t) = q_i\ |\ (e^TX^{(\lambda)}(t) - n)^+ = \ell, Z^{(\lambda)}(t) = z\big) \notag \\
=&\ \Prob\big(Q^{(\lambda)}_i(t) = q_i\ |\ (e^TX^{(\lambda)}(t) - n)^+ = \ell\big) \notag \\
=&  \Prob\big(B(\ell,p_i) =q_i \big). \label{eq:rig3}
\end{align}
Since \eqref{eq:rig3} holds for all $t \geq 0$, it holds in stationarity as well.

We now say a few words about how to construct $U^{(\lambda)}$ and
argue \eqref{eq:rig1}--\eqref{eq:rig3}. One would start with four
primitive sequences: a sequence of inter-arrival times, potential
service times, patience times, and routing decisions. The sequence of
potential service times would hold all the service information about
each customer provided they were patient enough to get into
service. The routing sequence would represent the phase each customer
is assigned upon entering the system.

To see why \eqref{eq:rig1} is true, we first observe that at any time
$t> 0$, the random variable $A(t)$ depends only on the inter-arrival
time primitives; in particular, it is independent of the routing
  sequence $\{\xi_j, j\ge 1\}$. Second, any customer to arrive after
customer number $B(t)=b$ has no impact on any of the servers at
any point in time during $[0,t]$. In particular, the primitives
including $\{\xi_{b+j}, j\ge 1\}$ associated to those customers
are independent of $B(t)=b$ and $Z^{(\lambda)}(t)$. Lastly, the
decisions of those customers whether to abandon or not
by time $t$ depends only on their arrival times, patience times, and
the service history in the interval $[0,t]$. In particular,
  the sequence $\{\zeta_{b+j}(t), j\ge 1\}$ is independent of
  $\{\xi_{b+j}, j\ge 1\}$. This proves the 
the first equality in \eqref{eq:rig1}.



We now move on to complete the proof of this lemma. We use \eqref{eq:rig3} to see that for any positive integer $N$, 
\begin{eqnarray}
&& \E\Big( [Q^{(\lambda)}_i(t) - p_i(e^TX^{(\lambda)}(t) - n)^+]^{2m}1_{\{ (e^T X^{(\lambda)}(t) - n)^+ \le N\}}\Big)\nonumber\\
&=& \sum \limits_{\ell=1}^{N}\E \Big{[}\big(B(\ell,p_i) - p_i \ell\big)^{2m}\Big{]}\Prob((e^TX^{(\lambda)}(t) - n) = \ell)\nonumber\\
&\leq&\sum \limits_{\ell=1}^{N}C(m,6) \ell^m\Prob((e^TX^{(\lambda)}(t) - n) = \ell)\nonumber\\
&=& C(m,6) \E \Big([(e^TX^{(\lambda)}(t) - n)^+]^m1_{\{ (e^T X^{(\lambda)}(t) - n)^+ \le N\}}\Big),\label{eq:sscproof}
\end{eqnarray}
where we have used the fact that there is a constant $C(m,6)>0$ such that
\begin{displaymath}
\E \Big{[}\big(B(\ell,p_i) - p_i \ell\big)^{2m}\Big{]} \le   
C(m,6) \ell^m \quad \text{ for all } \ell \ge 1;
\end{displaymath}
see, for example,  (4.10) of \cite{Knob2008}. Letting $t\to\infty$ in both sides of 
(\ref{eq:sscproof}), by the dominated convergence theorem, one has
\begin{eqnarray*}
&& \E\Big( [Q^{(\lambda)}_i(\infty) - p_i(e^TX^{(\lambda)}(\infty) - n)^+]^{2m}1_{\{ (e^T X^{(\lambda)}(\infty) - n)^+ \le N\}}\Big) \\
& &\le  C(m,6) \E \Big([(e^TX^{(\lambda)}(\infty) - n)^+]^m1_{\{ (e^T X^{(\lambda)}(\infty) - n)^+ \le N\}}\Big).
\end{eqnarray*}
Letting $N\to\infty$, by the monotone convergence theorem, one has 
\begin{eqnarray*}
\E (Q^{(\lambda)}_i(\infty) - p_i(e^TX^{(\lambda)}(\infty) - n)^+)^{2m} \leq C(m,6) \E \big[(e^TX^{(\lambda)}(\infty) - n)^+\big]^m.
\end{eqnarray*}
Then \eqref{eq:sscscaled} follows from this inequality for each $i$ and the fact that there is a constant $B_m>0$ such that  $\abs{x}^{2m}\le B_m \sum_{i=1}^d (x_i)^{2m}$ for all $x\in \R^d$. One can check that \eqref{eq:multinom} can be obtained by an argument very similar to the one used to prove \eqref{eq:sscscaled}.
\end{proof}

\begin{proof}[Proof of Lemma~\ref{lemma:SSC_bound}]
Recall that 
\begin{displaymath}
Z^{(\lambda)}(\infty) = X^{(\lambda)}(\infty) - Q^{(\lambda)}(\infty)
\end{displaymath} 
is the vector representing the number of customers of each type in service in steady-state. Then from \eqref{eq:interm_ssc} we have
\begin{eqnarray*}
&&\E \bigg[\partial_i f_h(\tilde X^{(\lambda)}(\infty))\big( \delta Q^{(\lambda)}_i(\infty) - p_i(e^T \tilde X^{(\lambda)}(\infty))^+\big) \bigg]\\
 &=& \sum_{k=1}^{d}\E \bigg[\partial_{ik} f_h(\xi)\big( \delta Q^{(\lambda)}_k(\infty) - p_k(e^T \tilde X^{(\lambda)}(\infty))^+\big)\big( \delta Q^{(\lambda)}_i(\infty) - p_i(e^T \tilde X^{(\lambda)}(\infty))^+\big)\bigg] \\
 &&+\ \E \bigg[\partial_i f_h\Big(\delta (Z^{(\lambda)}(\infty) - \gamma n) + p(e^T \tilde X^{(\lambda)}(\infty))^+\Big)\big( \delta Q^{(\lambda)}_i(\infty) - p_i(e^T \tilde X^{(\lambda)}(\infty))^+\big)\bigg].
\end{eqnarray*}
By Lemma~\ref{lemma:SSC}, the second expected value equals zero. For the first term, one can use the Cauchy-Schwarz inequality, together with the gradient bound \eqref{eq:gradbound3} and the SSC result \eqref{eq:sscscaled} to see that for all $1 \leq i,k \leq d$,
\begin{eqnarray*}
&&\E \bigg[\partial_{ik} f_h(\xi)\big( \delta Q^{(\lambda)}_k(\infty) - p_k(e^T \tilde X^{(\lambda)}(\infty))^+\big)\big( \delta Q^{(\lambda)}_i(\infty) - p_i(e^T \tilde X^{(\lambda)}(\infty))^+\big)\bigg] \\
&\leq & \sqrt{\E \Big[\big(\partial_{ik} f_h(\xi)\big)^2\Big] \sqrt{\E \bigg[\Big( \delta Q^{(\lambda)}_k(\infty) - p_k(e^T \tilde X^{(\lambda)}(\infty))^+\Big)^4\bigg] \E \bigg[\Big( \delta Q^{(\lambda)}_i(\infty) - p_i(e^T \tilde X^{(\lambda)}(\infty))^+\Big)^4	\bigg]}} \\
&\leq & \delta C(2, 3)\E \big[(e^T \tilde X^{(\lambda)}(\infty))^+\big]^2\sqrt{\E \Big[\big(\partial_{ik} f_h(\xi)\big)^2\Big]} \\
&\leq & \delta C(2, 3)\E \big[(e^T \tilde X^{(\lambda)}(\infty))^+\big]^2C(m,1)\sqrt{\E \Big[(1+\abs{\xi}^2)^2(1+\abs{\xi})^4\Big]}.
\end{eqnarray*}
We now combine everything together with the fact that $\xi$ satisfies \eqref{eq:xi_ineq} to conclude that there exists a constant $C(m,4)$ that does not depend on $\lambda$ or $n$, such that
\begin{eqnarray*}
&&\abs{\sum \limits_{i=1}^d \partial_i \E \bigg[f_h(\tilde X^{(\lambda)}(\infty))\Big{[}(\nu_i - \alpha-\sum \limits_{j=1}^d P_{ji}\nu_j)( \delta Q^{(\lambda)}_i(\infty) - p_i(e^T \tilde X^{(\lambda)}(\infty))^+)\Big{]}\bigg]} \notag \\
&\leq & \delta C(m, 4)\E \big[(e^T \tilde X^{(\lambda)}(\infty))^+\big]^2\sqrt{\E \Big[1 + \abs{\tilde X^{(\lambda)}(\infty)}^8\Big]},
\end{eqnarray*}
which concludes the proof of the lemma.
\end{proof}

\section{Proof of Theorem~\ref{thm:main}} \label{sec:thmpf}

To prove Theorem~\ref{thm:main}, we need an additional lemma on
uniform bounds for moments of scaled system size.  It will be proved
in Section \ref{app:CTMCmomentsproof}.
\begin{lemma}
\label{lemma:CTMCmoments}
For any integer $m \geq 0$, there exists a constant  $C(m, 5)>0$ (depending only on $(\beta,\alpha,p,\nu,P)$) such that
\begin{equation} \label{eq:CTMCunifmombound}
\E \abs{\tilde X^{(\lambda)}(\infty)}^{m} \leq C(m, 5).
\end{equation}
\end{lemma}

We remark that in the special case when the service time distribution
is taken to be hyper-exponential, it is proved in \cite{GamaStol2012} that
\begin{displaymath}
\limsup \limits_{\lambda \rightarrow \infty} \E \exp\Big(\theta \abs{\tilde X^{(\lambda)} (\infty)}\Big) < \infty
\end{displaymath}
for $\theta$ in some positive interval. The proof relies on a result that allows one to compare the
system with an infinite-server system, whose stationary distribution is  known to be Poisson.


\begin{proof}[Proof of Theorem~\ref{thm:main}]
  
It follows from Lemmas~\ref{lemma:diffbound} and \ref{lemma:SSC_bound} that

\begin{eqnarray}
&&\abs{\E h(\tilde X^{(\lambda)}(\infty)) - \E h(Y(\infty))} =
\abs{\E G_{U^{(\lambda)}} Af_h(U^{(\lambda)}(\infty)) - \E G_Y f_h(\tilde X^{(\lambda)}(\infty))}  \nonumber \\
&&\qquad  \leq \abs{\sum \limits_{i=1}^d  \E \bigg[\partial_i f_h(\tilde X^{(\lambda)}(\infty))\Big{[}(\nu_i - \alpha-\sum \limits_{j=1}^d P_{ji}\nu_j)( \delta Q^{(\lambda)}_i(\infty) - p_i(e^T \tilde X^{(\lambda)}(\infty))^+)\Big{]}\bigg]} \nonumber \\
&& \qquad \qquad +  \delta C(m,2)  \E \Big{[}(1+\abs{\tilde X^{(\lambda)}(\infty)}^2)^m(1 + \abs{\tilde X^{(\lambda)}(\infty)})^4\Big{]} \nonumber \\
&&\qquad  \leq \delta C(m, 4)\E \Big[\big((e^T \tilde X^{(\lambda)}(\infty))^+\big)^2\Big]\sqrt{\E \Big[1 + \abs{\tilde X^{(\lambda)}(\infty)}^8\Big]}\nonumber \\
&& \qquad \qquad +  \delta C(m,2)  \E \Big{[}(1+\abs{\tilde X^{(\lambda)}(\infty)}^2)^m(1 + \abs{\tilde X^{(\lambda)}(\infty)})^4\Big{]}
.\label{eq:proofbound}
\end{eqnarray}
By Lemma~\ref{lemma:CTMCmoments}, there are constants $B_1(m), B_2(m)>0$ (depending only on $(\beta,\alpha,p,\nu,P)$) such that 
\begin{eqnarray*}
&& \E \Big[\big((e^T \tilde X^{(\lambda)}(\infty))^+\big)^2\Big]\sqrt{\E \Big[1 + \abs{\tilde X^{(\lambda)}(\infty)}^8\Big]} \le B_1(m), \\
&& \E \Big{[}(1+\abs{\tilde X^{(\lambda)}(\infty)}^2)^m(1 + \abs{\tilde X^{(\lambda)}(\infty)})^4\Big{]}\le B_2(m).
\end{eqnarray*}
Therefore, the right side of (\ref{eq:proofbound}) is less than or equal to 
\begin{eqnarray*}
\lefteqn{ \delta C(m,4) B_1(m) + \delta C(m,2) B_2(m) }\\
&\leq&  \Big(C(m,4) B_1(m) +  C(m,2) B_2(m) \Bigr) \frac{1}{\sqrt{\lambda}} \quad \text{ for }\lambda > 0.
\end{eqnarray*}
 This concludes the proof of Theorem~\ref{thm:main}.
\end{proof}

\section*{Acknowledgements} 

The authors thank Shuangchi He, Josh Reed and John Pike for stimulating
discussions. They also thank the participants of Applied Probability
\& Risk Seminar in Fall 2014 at Columbia University for their feedback
on this research.  This research is supported in part by NSF Grants
CNS-1248117 and CMMI-1335724.

\appendix
\section{Proofs} \label{app:proofs}

\subsection{Proof of Lemma~\ref{lemma:poisson} (Gradient Bounds)}
\label{sec:proof-lemma-refl}
Before proving the lemma, we first state the common quadratic Lyapunov
function introduced in \cite{DiekGao2013}.  This Lyapunov function
plays a key role in our paper. As in (5.24) of \cite{DiekGao2013}, for
$x \in \R^d$, define
\begin{equation} \label{eq:deflyapou}
V(x) = (e^Tx)^2 + \kappa[x-p \phi(e^Tx)]' M [x-p \phi(e^Tx)],
\end{equation}
where $\kappa>0$ is some constant, $M$ is some $d\times d$ positive definite matrix, and the function $\phi$ is a smooth approximation to $x \longmapsto x^+$ and is defined by
\begin{displaymath}
\phi(x) =  \begin{cases}
    x, & \text{if $x \geq 0$},\\
    -\frac{1}{2}\epsilon, & \text{if $x \leq -\epsilon$},\\
    \text{smooth}, & \text{if $-\epsilon < x < 0$}.
  \end{cases}
\end{displaymath}
In (5.24) of \cite{DiekGao2013}, the authors use $\tilde Q$ to represent the positive definite matrix that we called $M$ in \eqref{eq:deflyapou}. We use $M$ instead of $\tilde Q$ on purpose, to avoid any potential confusion with the queue size $Q(t)$. 
For our purposes, ``smooth" means that $\phi$ can be anything as long as $\phi \in C^3(\R^d)$.  We require that the ``smooth" part of $\phi$ also satisfies $-\frac{1}{2}\epsilon < \phi (x) < x$ and $0 \leq {\phi}'(x)\leq 1$. For example, $\phi$ can be taken to be a polynomial of sufficiently high degree on $(-\epsilon, 0)$ and this will satisfy our requirements.  The vector $p$ is as in (\ref{eq:defou}). The constant $\kappa$ and matrix $M$ are chosen just as in \cite{DiekGao2013}; their exact values are not important to us. In their paper, they show that $V$ satisfies
\begin{displaymath}
G_YV(x) \leq -c_1V(x) + c_2 \quad \text{ for all } x\in \R^d
\end{displaymath}
for some positive constants $c_1$,$c_2$; this result requires $\alpha > 0$, i.e.\ a strictly positive abandonment rate. Before proceeding to the proof of Lemma~\ref{lemma:poisson}, we state two bounds on $V$ that shall be useful in the future. For some constant $C>0$,
\begin{eqnarray} 
V(x) \leq C(1+ \abs{x}^2), \label{eq:lyapupperbound}\\
\abs{x}^2 \leq C(1+V(x)). \label{eq:lyaplowbound}
\end{eqnarray}
The first is immediate from the form of $V$, while the second is proved in \cite{DiekGao2013}.

\begin{proof}[Proof of Lemma~\ref{lemma:poisson}]
Without loss of generality, we may assume that $h(0) = 0$, otherwise one may consider $h(x) -h(0)$. This lemma is essentially a restatement of equation (22) and equation (40) from the discussion that follows after \cite[Theorem 4.1]{Gurv2014}. We verify that (22) and (40) are applicable in our case by first confirming that we have a function satisfying assumption 3.1 of \cite{Gurv2014}. Recalling the definition of $V$ from (\ref{eq:deflyapou}), when $\phi$ is taken to be a polynomial (of sufficiently high degree to guarantee $V \in C^3(\R^d)$), the function 
\begin{displaymath}
1+V(x)
\end{displaymath}
satisfies assumption 3.1. To verify condition (17) of Assumption 3.1, one observes that 
\begin{displaymath}
X^{(\lambda)}(t) \leq X^{(\lambda)}(0) + n + A^{(\lambda)}(t),
\end{displaymath}
where $A^{(\lambda)}(t)$ is the total number of arrivals to the system by time $t$
and it is a Poisson random variable with mean $\lambda t$ for each
$t\ge 0$. The properties of Poisson processes then yield (17). By
\cite[Remark 3.4]{Gurv2014},
\begin{displaymath}
C(1+V(x))^m
\end{displaymath}
also satisfies assumption 3.1 for any constant $C>0$. Since we require that $\abs{h(x)} \leq \abs{x}^m$, by (\ref{eq:lyaplowbound}) we have 
\begin{displaymath}
\abs{h(x) - \E h(Y(\infty))} \leq \abs{x}^m + \E \abs{Y(\infty)}^m \leq C_m (1+V(x))^m.
\end{displaymath}
The finiteness of $\E \abs{Y(\infty)}^m$ is guaranteed because one of the conditions of assumption 3.1 is that
\begin{displaymath}
G_Y (1+V(x))^m \leq -c_1 (1+V(x))^m + c_2
\end{displaymath}
for some positive constants $c_1$ and $c_2$.\ Therefore, equation (22) gives us (\ref{eq:gradbound1}) and equation (40) gives us (\ref{eq:gradbound2}) and (\ref{eq:gradbound3}). We get (\ref{eq:gradbound4}) by observing that in the discussion preceding (40), everything still holds if we replace $B_x(\bar l / \sqrt{n})$ by an open ball of radius $1$ centered at $x$. We wish to point out that the constants in (40) and (22) do not depend on the choice of function $h$.
\end{proof}

\subsection{Proof of Lemma~\ref{lemma:diffbound} (Generator Difference)} \label{app:taylorexp}
The main idea here is that $G_Y f_h(x)$ is
hidden within $G_{U^{(\lambda)}} Af_h(u)$, where the lifting operator $A$ is in (\ref{eq:lifter}). We algebraically manipulate the Taylor expansion of $G_{U^{(\lambda)}} Af_h(u)$ to make this evident. First, we first rearrange the terms in the Taylor expansion (\ref{eq:GXtaylor}) to group them by partial derivatives. Thus, $G_{U^{(\lambda)}} Af_h(u)$ equals 
\begin{eqnarray*}
&&\sum \limits_{i=1}^d \delta \partial_i f_h(x)\Big{[}p_i \lambda - \alpha q_i - \nu_i z_i + \sum \limits_{j=1}^d P_{ji}\nu_j z_j\Big{]} \\
&&+ \sum \limits_{i=1}^d \frac{\delta^2}{2}\partial_{ii}f_h(x)\Big{[}p_i \lambda + \alpha q_i + \nu_i z_i + \sum \limits_{j=1}^d P_{ji}\nu_j z_j\Big{]} - \sum \limits_{i \neq j}^d \delta^2\partial_{ij}f_h(x)\big{[}P_{ij}\nu_i z_i \big{]} \\
&&+ \sum \limits_{i=1}^d \frac{\delta^2}{2} \Big(\partial_{ii}f_h(\xi_{i}^-)-\partial_{ii}f_h(x)\Big)\Big{[}\alpha q_i + (1-\sum \limits_{j=1}^d P_{ij})\nu_i z_i\Big{]} \\
&&+ \sum \limits_{i=1}^d \frac{\delta^2}{2} \Big(\partial_{ii}f_h(\xi_{i}^+)-\partial_{ii}f_h(x)\Big)\big{[}\lambda p_i \big{]}- \sum \limits_{i \neq j}^d \delta^2 \Big(\partial_{ij}f_h(\xi_{ij})-\partial_{ij}f_h(x)\Big)\big{[}P_{ij}\nu_i z_i\big{]}\\
&&+ \sum \limits_{i=1}^d \sum \limits_{j=1}^d \frac{\delta^2}{2} \Big(\partial_{ii}f_h(\xi_{ij})-\partial_{ii}f_h(x)\Big)\Big{[} P_{ij}\nu_i z_i +  P_{ji}\nu_j z_j\Big{]} .
\end{eqnarray*}
To proceed we observe that (\ref{eq:defph}) gives us the identity
\begin{equation}\label{eq:taylorexpidentity}
-\nu_i \gamma_i n + \sum \limits_{j=1}^d P_{ji}\nu_j\gamma_j n  = -n p_i. 
\end{equation}
Recall the form of $G_Y f_h(x)$ from (\ref{eq:DMgen}). From the form of $\Sigma$ in (\ref{eq:OUdiffusioncoeff}), we see that
\begin{equation} \label{eq:OUaltdiffusioncoef}
\Sigma_{ii} = 2\Big(p_i + \sum \limits_{j=1}^d P_{ji}\gamma_j \nu_j\Big), \quad \Sigma_{ij} = -(P_{ij}\nu_i \gamma_i + P_{ji}\nu_j\gamma_j) \text{ for $j\neq i$}
\end{equation}
 using ($\ref{eq:DMgen}$), ($\ref{eq:taylorexpidentity}$) and ($\ref{eq:OUaltdiffusioncoef}$), the difference $G_{U^{(\lambda)}} Af_h(u) - G_Y f_h(x)$ becomes

\begin{eqnarray}
&&\sum \limits_{i=1}^d \partial_i f_h(x)\Big{[}(\nu_i - \alpha-\sum \limits_{j=1}^d P_{ji}\nu_j)( \delta q_i - p_i(e^T x)^+)\Big{]}  \label{eq:errorterm} \\
&&+ \sum \limits_{i=1}^d \partial_{ii}f_h(x)\Big{[}\sum \limits_{j=1}^d P_{ji}\nu_j\gamma_j \Big{]}(n \delta^2 - 1) - \sum \limits_{i \neq j}^d \partial_{ij}f_h(x)\Big{[}P_{ij}\nu_i\gamma_i + P_{ji} \nu_j\gamma_j\Big{]}(n\delta^2 - 1) \notag \\
&&- \sum \limits_{i=1}^d \frac{\delta^2}{2}\partial_{ii}f_h(x)\Big{[}p_i (\lambda - n)- \alpha  q_i - \nu_i (z_i- \gamma_i n) - \sum \limits_{j=1}^d P_{ji}\nu_j (z_j- \gamma_j n)\Big{]} \notag \\
&&- \sum \limits_{i \neq j}^d \frac{\delta^2}{2}\partial_{ij}f_h(x)\Big{[}P_{ij}\nu_i (z_i- \gamma_i n) + P_{ji} \nu_j (z_j- \gamma_j n)\Big{]} \notag \\
&&+ \sum \limits_{i=1}^d \frac{\delta^2}{2} (\partial_{ii}f_h(\xi_{i}^-)-\partial_{ii}f_h(x))\Big{[}\alpha q_i + (1-\sum \limits_{j=1}^d P_{ij})\nu_i z_i\Big{]}\notag \\
&&  + \sum \limits_{i=1}^d \frac{\delta^2}{2} (\partial_{ii}f_h(\xi_{i}^+)-\partial_{ii}f_h(x))\Big{[}\lambda p_i \Big{]}- \sum \limits_{i \neq j}^d \delta^2 (\partial_{ij}f_h(\xi_{ij})-\partial_{ij}f_h(x))\Big{[}P_{ij}\nu_i z_i\Big{]} \notag \\
&&+ \sum \limits_{i=1}^d \sum \limits_{j=1}^d \frac{\delta^2}{2} (\partial_{ii}f_h(\xi_{ij})-\partial_{ii}f_h(x))\Big{[} P_{ij}\nu_i z_i +  P_{ji}\nu_j z_j\Big{]} .  \notag 
\end{eqnarray}
We remind the reader that our target is to prove that
  \begin{eqnarray*}
&&G_{U^{(\lambda)}} Af_h(u) - G_Y f_h(x) \nonumber\\
&=&  \sum \limits_{i=1}^d \partial_i f_h(x)\Big{[}(\nu_i - \alpha-\sum \limits_{j=1}^d P_{ji}\nu_j)( \delta q_i - p_i(e^T x)^+)\Big{]} + E(u),
  \end{eqnarray*}
  where $E(u)$ is an error term that satisfies 
\begin{displaymath}
\abs{E(u)} \leq \delta\, C(m,2) (1+\abs{x}^2)^m (1+\abs{x})^4.
\end{displaymath}
%
We choose $E(u)$ to be all the terms in \eqref{eq:errorterm} except for the first line. We now describe how to bound $\abs{E(u)}$.
Most of the summands in (\ref{eq:errorterm}) look as follows: a term in large square brackets multiplied by some partial derivative of $f_h$. The partial derivatives are very easy to bound; we simply use (\ref{eq:gradbound2}) - (\ref{eq:gradbound4}). We wish to point out that $\xi_{i}^+$, $\xi_{i}^-$ and $\xi_{ij}$ lie within distance $2\delta$ of $x$. When $2\delta < 1$, (\ref{eq:gradbound4}) implies 
\begin{equation} \label{eq:lg4cond}
\abs{\partial_{ij}f_h(\xi)-\partial_{ij}f_h(x)} \leq 2\delta C (1+ \abs{x}^2)^m (1+ \abs{x})^3
\end{equation} 
for some constant $C>0$ (i.e.\ an extra $\delta$ term is gained). When $2\delta \geq 1$ (by Remark~\ref{rem:lambdarange} this occurs in finitely many cases), we may use (\ref{eq:gradbound3}) to obtain (\ref{eq:lg4cond}) with a redefined $C$. From here on out, we shall let $C>0$ be a generic positive constant that will change from line to line, but will always be independent of $\lambda$ and $n$.

Now we shall list the facts needed to bound all the square bracket terms in (\ref{eq:errorterm}) except for the very first one. Recall that we are operating in the Halfin-Whitt regime as defined by \eqref{eq:square-root}. Therefore, 
\begin{displaymath}
(n\delta^2 - 1) = \delta \beta \text{ and } \delta (\lambda - n) = -\beta.
\end{displaymath}
 Furthermore, it must be true that 
\begin{displaymath}
 \delta q_i \leq (e^T x)^+ \leq C \abs{x}, \label{eq:qibound}
\end{displaymath}
as the number of phase $i$ customers may never exceed the total queue size. Next, 
\begin{displaymath}
\abs{\delta(z_i - \gamma_i n)} = \abs{x_i - \delta q_i} \leq C \abs{x}
\end{displaymath}
and lastly, 
\begin{displaymath}
\abs{\delta^2 z_i} \leq \abs{\delta^2 \gamma_i n} + \abs{\delta^2 (z_i - \gamma_i n)} \leq C(1+ \abs{x}).
\end{displaymath}
It is now a simple matter to verify that the inequalities above, combined with the bounds on the partials of $f_h$ are all that it takes to achieve our desired upper bound.

%

\subsection{Proof of Lemma~\ref{lemma:CTMCmoments} (Moment Bounds)} \label{app:CTMCmomentsproof}

We first provide an intuitive roadmap for the proof. The goal is to show that a Lyapunov function for the diffusion process is also a Lyapunov function for the CTMC; this has two parts to it. In the first part of this proof, we compare how the two generators $G_{U^{(\lambda)}}$ and $G_Y$ act on this Lyapunov function, obtaining an upper bound for the difference $G_{U^{(\lambda)}} - G_Y$ in (\ref{eq:CTMCmomlemmaeq2}).
One notes  that the right hand side of (\ref{eq:CTMCmomlemmaeq2}) is unbounded. This is due to the difference in dimensions of the CTMC and diffusion process. To overcome this difficulty, we move on to the second part of the proof, which exploits our SSC result in Lemma~\ref{lemma:SSC} to bound the expectation of the right hand side of (\ref{eq:CTMCmomlemmaeq2}). We end up with a recursive relationship that guarantees the $2m$th moment is bounded (uniformly in $\lambda$ and $n$ satisfying \eqref{eq:square-root}) provided that the $m$th moment is. Finally, we rely on prior results obtained in \cite{DaiDiekGao2014} for a uniform bound on the first moment.

We remark that a version of this lemma was already proved \cite[Theorem 3.3]{Gurv2014} for the case where the dimension of the CTMC equals the dimension of the diffusion process. However, the difference in dimensions poses an additional technical challenge, which is overcome in the second part of this proof.

Its enough to prove (\ref{eq:CTMCunifmombound}) for the cases when $m=2^j$ for some $j\geq 0$. Furthermore, we may assume that $\lambda \geq 4$ because by Remark~\ref{rem:lambdarange}, there are only finitely many cases when $\lambda < 4$. In all those cases, $\E \abs{\tilde X^{(\lambda)}(\infty)}^m < \infty$ by (\ref{eq:CTMCmgfexist}). Throughout the proof, we shall use $C,C_1,C_2,C_3,C_4$ to denote generic positive constants that may change from line to line. They may depend on $(m,\beta,\alpha,p,\nu,P)$, but will be independent of both $\lambda$ and $n$. Define
\begin{displaymath}
V_m(x) = (1+V(x))^m,
\end{displaymath}
where $V$ is as in (\ref{eq:deflyapou}). By \cite[Remark 3.4]{Gurv2014}, $V_m$ also satisfies
\begin{displaymath}
G_Y V_m(x) \leq -C_1 V_m(x) + C_2
\end{displaymath}
as long as $V \in C^3(\R^d)$ and satisfies condition (30) of \cite{Gurv2014}, which is easy to verify. To prove the lemma, we will show that for large enough $\lambda$, $V$ satisfies 
\begin{displaymath}
\E G_{U^{(\lambda)}} AV_m(U^{(\lambda)}(\infty)) \leq - C_1 \E  V_m(\tilde X^{(\lambda)}(\infty)) + C_2,
\end{displaymath}
where $A$ is the lifting operator defined in (\ref{eq:lifter}). 
 We begin by observing
\begin{equation} \label{eq:CTMCmomlemmaeq1}
G_{U^{(\lambda)}} AV_m \leq G_{U^{(\lambda)}}AV_m - G_Y V_m + G_Y V_m \leq G_{U^{(\lambda)}}AV_m - G_Y V_m -C_1 V_m + C_2.
\end{equation}
Using (\ref{eq:errorterm}), we write $G_{U^{(\lambda)}}AV_m - G_Y V_m$ as
\begin{eqnarray*}
&&\sum \limits_{i=1}^d \partial_i V_m(x)\Big{[}(\nu_i - \alpha-\sum \limits_{j=1}^d P_{ji}\nu_j)( \delta q_i - p_i(e^T x)^+)\Big{]} \\
&& + \sum \limits_{i=1}^d \partial_{ii}V_m(x)\Big{[}\sum \limits_{j=1}^d P_{ji}\nu_j\gamma_j \Big{]}(n \delta^2 - 1) - \sum \limits_{i \neq j}^d \partial_{ij}V_m(x)\Big{[}P_{ij}\nu_i\gamma_i + P_{ji} \nu_j\gamma_j\Big{]}(n\delta^2 - 1) \notag \\
&&- \sum \limits_{i=1}^d \frac{\delta^2}{2}\partial_{ii}V_m(x)\Big{[}p_i (\lambda - n)- \alpha  q_i - \nu_i (z_i- \gamma_i n) - \sum \limits_{j=1}^d P_{ji}\nu_j (z_j- \gamma_j n)\Big{]} \notag \\
&&- \sum \limits_{i \neq j}^d \frac{\delta^2}{2}\partial_{ij}V_m(x)\Big{[}P_{ij}\nu_i (z_i- \gamma_i n) + P_{ji} \nu_j (z_j- \gamma_j n)\Big{]} \notag \\
&&+ \sum \limits_{i=1}^d \frac{\delta^2}{2} (\partial_{ii}V_m(\xi_{i}^-)-\partial_{ii}V_m(x))\Big{[}\alpha q_i + (1-\sum \limits_{j=1}^d P_{ij})\nu_i z_i\Big{]}  \notag \\
&&+ \sum \limits_{i=1}^d \frac{\delta^2}{2} (\partial_{ii}V_m(\xi_{i}^+)-\partial_{ii}V_m(x))\Big{[}\lambda p_i \Big{]}- \sum \limits_{i \neq j}^d \delta^2 (\partial_{ij}V_m(\xi_{ij})-\partial_{ij}V_m(x))\Big{[}P_{ij}\nu_i z_i\Big{]} \notag \\
&&+ \sum \limits_{i=1}^d \sum \limits_{j=1}^d \frac{\delta^2}{2} (\partial_{ii}V_m(\xi_{ij})-\partial_{ii}V_m(x))\Big{[} P_{ij}\nu_i z_i +  P_{ji}\nu_j z_j\Big{]} .  \notag 
\end{eqnarray*}
Now we wish to bound the derivatives of $V_m$. By \cite[Remark 3.4]{Gurv2014}, $V_m$ satisfies (16) and (30) of \cite{Gurv2014}, namely

\begin{equation}
\sup \limits_{\abs{y} \leq 1} \frac{V_m(x+y)}{V_m(x)} \leq C \label{eq:Vsubexp}
\end{equation}
and
\begin{equation} \label{eq:cond30gurvich}
(\abs{\partial_i V_m(x)} + \abs{\partial_{ij} V_m(x)} +\abs{\partial_{ijk} V_m(x)})(1+\abs{x}) \leq C V_m(x).
\end{equation}
For $\xi$ being one of $\xi_{i}^+$, $\xi_{i}^-$ or $\xi_{ij}$,
\begin{equation} \label{eq:VD3bound}
\abs{\partial_{ij}V_m(\xi) - \partial_{ij}V_m(x)}(1+\abs{x}) \leq \delta \abs{\partial_{iji} V_m(\eta) +\partial_{ijj}V_m(\eta) }(1+\abs{x}) \leq C \delta V_m(x),
\end{equation}
where the first inequality comes from a Taylor expansion and the second inequality follows by (\ref{eq:cond30gurvich}), the fact that $\abs{\eta - x} \leq 2\delta < 1$ and by (\ref{eq:Vsubexp}).
Following the exact same argument that we used to bound (\ref{eq:errorterm}) in the proof of Lemma~\ref{lemma:diffbound} (with (\ref{eq:cond30gurvich}) and (\ref{eq:VD3bound}) replacing the gradient bounds of $f_h$ there), we get 
\begin{displaymath}
G_{U^{(\lambda)}}AV_m - G_Y V_m \leq  C \delta V_m(x) +  C\sum \limits_{i=1}^d \abs{\partial_i V_m(x)}\Big{[}\abs{ q_i - p_i(e^T x)^+}\Big{]}.
\end{displaymath}
Differentiating $V$, we see that
\begin{displaymath}
(\nabla V(x))^T = 2(e^T x)e^T + 2 \kappa (x^T - p^T \phi(e^T x)) \tilde Q (I - p e^T \phi'(e^Tx)).
\end{displaymath} 
Combined with the fact that $0 \leq \phi'(x)\leq 1$, it is clear that
\begin{displaymath}
\abs{\partial_i V(x)} \leq C(1+\abs{x}).
\end{displaymath}
Therefore,
\begin{equation}\label{eq:CTMCmomlemmaeq2}
G_{U^{(\lambda)}}AV_m - G_Y V_m \leq  C \delta V_m(x) +  C\sum \limits_{i=1}^d mV_{m-1}(x)(1+\abs{x})\Big{[}\abs{ q_i - p_i(e^T x)^+}\Big{]} .
\end{equation}
It remains to find an appropriate bound for 
\begin{displaymath}
V_{m-1}(x) (1+\abs{x})\Big{[}\abs{ q_i - p_i(e^T x)^+}\Big{]} = \delta V_{m-1}(x)(1+\abs{x})\Bigg{[}\frac{\abs{q_i - p_i(e^T x)^+}}{\delta}\Bigg{]}.
\end{displaymath}
We have
\begin{eqnarray}
&&\delta V_{m-1}(x) (1+\abs{x})\Bigg{[}\frac{\abs{q_i - p_i(e^T x)^+}}{\delta}\Bigg{]} \notag \\
&\leq & \sqrt{\delta}V_{m-1}(x)(1+\abs{x})^2 + \sqrt{\delta}V_{m-1}(x)\Bigg{[}\frac{\abs{q_i - p_i(e^T x)^+}^2}{\delta}\Bigg{]} \notag \\
&\leq & C\sqrt{\delta}V_{m}(x) + \sqrt{\delta}V_{m-2}(x)V_2(x)+ \sqrt{\delta}V_{m-2}(x)\Bigg{[}\frac{\abs{q_i - p_i(e^T x)^+}^2}{\delta}\Bigg{]}^2 \notag \\
&\leq & C\sqrt{\delta}V_{m}(x) + \sqrt{\delta}V_{m}(x)+ \sqrt{\delta}V_{m-4}(x)V_4(x)+ \sqrt{\delta} V_{m-4}(x)\Bigg{[}\frac{\abs{q_i - p_i(e^T x)^+}^2}{\delta}\Bigg{]}^4 \notag \\ 
&\leq & \ldots  \notag \\
&\leq & C\sqrt{\delta}V_{m}(x) + \sqrt{\delta}\Bigg{[}\frac{\abs{q_i - p_i(e^T x)^+}^2}{\delta}\Bigg{]}^m, \label{eq:CTMCmomlemmaeq3}
\end{eqnarray}
where in the last inequality, we used the fact that $m = 2^j$. Using (\ref{eq:CTMCmomlemmaeq1}), (\ref{eq:CTMCmomlemmaeq2}) and (\ref{eq:CTMCmomlemmaeq3}), 
\begin{displaymath}
G_{U^{(\lambda)}} AV_m(u) \leq - V_m(x)(C_1 - \sqrt{\delta}C_3) + C_2 + \sqrt{\delta}C_4\sum \limits_{i=1}^d \Bigg{[}\frac{\abs{q_i - p_i(e^T x)^+}^2}{\delta}\Bigg{]}^m,
\end{displaymath}
where $x$ and $q$ are related to $u$ by (\ref{eq:projlittle}).
The arguments in the proof of Lemma~\ref{lemma:CTMCbar} can be used to show 
\begin{displaymath}
\E G_{U^{(\lambda)}} A V_m(U^{(\lambda)}(\infty)) = 0.
\end{displaymath} 
Therefore, for $\delta$ small enough, 
\begin{eqnarray*}
E\abs{\tilde X^{(\lambda)}(\infty)}^{2m} &\leq & C\E V_m(\tilde X^{(\lambda)}(\infty))\\
 &\leq & \frac{C}{(C_1 - \sqrt{\delta}C_3)}\Bigg{(}C_2 + \sqrt{\delta}C_4\sum \limits_{i=1}^d \frac{\E \abs{\delta Q^{(\lambda)}_i(\infty) - p_i(e^T \tilde X^{(\lambda)}(\infty))^+}^{2m}}{\delta^m}\Bigg{)}.
\end{eqnarray*}
By (\ref{eq:sscscaled}), it follows that 
\begin{displaymath}
\E \abs{\tilde  X^{(\lambda)}(\infty)}^{2m} \leq \frac{C}{C_1 - \sqrt{\delta}C_3} \Bigg{(}1 + \sqrt{\delta}\E [(e^T \tilde X^{(\lambda)}(\infty))^+]^{m}\Bigg{)}.
\end{displaymath}
Hence, we have a recursive relationship that guarantees 
\begin{displaymath}
\sup \limits_{\lambda > 0} \E \abs{\tilde  X^{(\lambda)}(\infty)}^{2m} < \infty
\end{displaymath}
whenever 
\begin{displaymath}
\sup \limits_{\lambda > 0} \E[(e^T\tilde  X^{(\lambda)}(\infty))^+]^m < \infty.
\end{displaymath}
To conclude, we need to verify that
\begin{displaymath}
\sup \limits_{\lambda > 0} \E[(e^T\tilde  X^{(\lambda)}(\infty))^+] < \infty,
\end{displaymath}
but this was proved in equation (5.2) of \cite{DaiDiekGao2014}.

\bibliography{dai02282015}
\end{document}